\pdfoutput=1
\documentclass[a4paper]{amsart}
 \usepackage{
  amsmath, amscd, amssymb, mathrsfs, 
  mathtools, fullpage, enumerate, thmtools,
  mdframed, fancyhdr, tikz-cd}

\usepackage[pagebackref, hypertexnames=false]{hyperref}
\usepackage[capitalise]{cleveref}

\newenvironment{smatrix}{\left( \begin{smallmatrix} } {\end{smallmatrix} \right) }

\newcommand{\stbt}[4]{\begin{smatrix}#1 & #2 \\ #3 & #4\end{smatrix}}
\newcommand{\dfour}[4]{\begin{smatrix}#1\\ &#2 \\ &&#3 \\ &&& #4\end{smatrix}}

\theoremstyle{plain}
\newtheorem{theorem}{Theorem}[subsection]
\newtheorem{introtheorem}{Theorem}

\newtheorem{lemma}[theorem]{Lemma}
\newtheorem{proposition}[theorem]{Proposition}
\newtheorem{corollary}[theorem]{Corollary}
\newtheorem{definition}[theorem]{Definition}
\newtheorem{notation}[theorem]{Notation}

\newtheorem{assumption}[theorem]{Assumption}
\newtheorem*{convention}{Convention}

\theoremstyle{remark}
\declaretheorem[name=Remark,sibling=theorem,qed={\lower-0.3ex\hbox{$\diamond$}}]{remark}
\declaretheorem[name=Note,sibling=theorem,qed={\lower-0.3ex\hbox{$\diamond$}}]{note}

\DeclareMathOperator{\AJ}{AJ}
\DeclareMathOperator{\Fil}{Fil}
\DeclareMathOperator{\GL}{GL}
\DeclareMathOperator{\GSp}{GSp}
\DeclareMathOperator{\Gal}{Gal}
\DeclareMathOperator{\Gr}{Gr}
\DeclareMathOperator{\Hom}{Hom}
\DeclareMathOperator{\Ind}{Ind}
\DeclareMathOperator{\Iw}{Iw}
\DeclareMathOperator{\Kl}{Kl}
\DeclareMathOperator{\LE}{\mathcal{LE}}
\DeclareMathOperator{\Mot}{mot}
\DeclareMathOperator{\Sieg}{Si}
\DeclareMathOperator{\Sym}{Sym}
\DeclareMathOperator{\alg}{alg}
\newcommand{\crit}{\mathrm{crit}}
\DeclareMathOperator{\diag}{diag}
\DeclareMathOperator{\mot}{mot}
\DeclareMathOperator{\pr}{pr}
\DeclareMathOperator{\vol}{vol}

\renewcommand{\AA}{\mathbf{A}}

\newcommand{\Af}{\AA_{\mathrm{f}}}
\newcommand{\CC}{\mathbf{C}}
\newcommand{\Eis}{\mathrm{Eis}}
\newcommand{\Pif}{\Pi_{\mathrm{f}}}
\newcommand{\QQbar}{\overline{\QQ}}
\newcommand{\QQ}{\mathbf{Q}}
\newcommand{\Ql}{\QQ_\ell}
\newcommand{\Qp}{\QQ_p}
\newcommand{\ZZ}{\mathbf{Z}}
\newcommand{\Zp}{\ZZ_p}
\newcommand{\bfj}{\mathbf{j}}
\newcommand{\cB}{\mathcal{B}}
\newcommand{\cD}{\mathcal{D}}
\newcommand{\cE}{\mathcal{E}}
\newcommand{\cH}{\mathcal{H}}
\newcommand{\cL}{\mathcal{L}}
\newcommand{\cM}{\mathcal{M}}
\newcommand{\cO}{\mathcal{O}}
\newcommand{\cQ}{\mathcal{Q}}
\newcommand{\cS}{\mathcal{S}}
\newcommand{\cV}{\mathcal{V}}
\newcommand{\cW}{\mathcal{W}}
\newcommand{\can}{\mathrm{can}}
\newcommand{\ch}{\mathrm{ch}}
\newcommand{\dep}{\mathrm{dep}}
\newcommand{\dR}{\mathrm{dR}}
\newcommand{\sph}{\mathrm{sph}}
\newcommand{\et}{\text{\textup{\'et}}}
\newcommand{\id}{\mathrm{id}}
\newcommand{\into}{\hookrightarrow}
\newcommand{\uPhi}{\underline{\Phi}}
\newcommand{\uchi}{\underline{\chi}}
\newcommand{\wZ}{\widetilde{Z}}

\newcommand{\Dcris}{\mathbf{D}_{\mathrm{cris}}}
\newcommand{\DdR}{\mathbf{D}_{\dR}}

\numberwithin{equation}{subsection}
\renewcommand{\le}{\leq}
\renewcommand{\leq}{\leqslant}
\renewcommand{\ge}{\geq}
\renewcommand{\geq}{\geqslant}

\author{David Loeffler}
\address[Loeffler]{UniDistance Suisse, Schinerstrasse 18, 3900 Brig, Switzerland}
\email{david.loeffler@unidistance.ch}
\urladdr{\href{http://orcid.org/0000-0001-9069-1877}{0000-0001-9069-1877}}

\author{Sarah Livia Zerbes}
\address[Zerbes]{Department of Mathematics, ETH Z\"urich, R\"amistrasse 101, 8092 Z\"urich, Switzerland}
\email{sarah.zerbes@math.ethz.ch}
\urladdr{\href{http://orcid.org/0000-0001-8650-9622}{0000-0001-8650-9622}}

\thanks{D.L. gratefully acknowledges the support of the European Research Council through the Horizon 2020 Excellent Science programme (Consolidator Grant ``ShimBSD: Shimura varieties and the BSD conjecture'', grant ID 101001051)}

\title{A universal Euler system for $\GSp_4$}

\begin{document}
\renewcommand{\crefrangeconjunction}{--} 


\begin{abstract}
 In our earlier work with Christopher Skinner \cite{LSZ17} we constructed Euler systems for the 4-dimensional spin Galois representations corresponding to automorphic forms for $\GSp_4$. This construction depended on various arbitrary choices of local test data. In this paper, we use multiplicity-one results for smooth representations to determine how these Euler system classes depend on the choice of test data, showing that all of these classes lie in a 1-dimensional space and are explicit multiples (given by local zeta-integrals) of a ``universal'' class independent of the choice of test data.
\end{abstract}

\maketitle


\section{Introduction}

 \subsection{Overview}
 
  This paper is intended as a complement to our recent work \cite{LSZ17} with Christopher Skinner, in which we constructed an Euler system for four-dimensional Galois representations arising from cuspidal automorphic representations of $\GSp_4$. The construction of this Euler system depends on certain choices of auxiliary local data (``test data''). In \emph{op.cit.} we simply worked with an arbitrary fixed choice of these data, but it is far from obvious \emph{a priori} what the best choice should be.

  In this note, we show that different choices of these data only affect the Euler system by a scaling factor, and that these scaling factors are explicitly given by local zeta integrals.

  \subsubsection*{Note} 
  
   The results in this paper are among the inputs in the paper \cite{LZ26}, which proves a reciprocity law relating the $\GSp_4$ Euler system to critical $L$-values, and applies this to deduce cases of the Bloch--Kato and Iwasawa main conjectures. The material in this paper was originally part of the earlier drafts of \cite{LZ26}, but since it uses rather different methods from the proof of the reciprocity law, and the combined paper had become excessively long, we have factored out these representation-theoretic arguments into the present separate paper. We hope these methods will also be of interest for other Euler system constructions beyond the scope of $\GSp_4$.

 \subsection{Main results}

  We briefly recall the relevant constructions from \cite{LSZ17}. Let $r_1 \ge r_2 \ge 0$ be integers, and choose $q, r$ with $0 \le q \le r_2$ and $0 \le r \le r_1 - r_2$. In this introduction we suppose $q < r_2$, which simplifies the statements slightly. In Definition 8.3.1 of \cite{LSZ17}, we defined a map
  \[ \mathcal{LE}^{[q,r]} : \cS(\Af^4; \QQ) \otimes \cH(G(\Af); \QQ) \to H^4_{\mot}\left(Y_G, \cD_{\QQ}(-q)\right), \]
  where $Y_G$ is the Shimura variety for $\GSp_4$, and $\cD$ is a relative motive over $Y_G$ corresponding to an algebraic representation of $G$ of highest weight $(r_1, r_2)$. Here $\cS(\Af^4; \QQ)$ is the space of $\QQ$-valued Schwartz functions on $\Af^4$, and $\cH(G(\Af); \QQ)$ the Hecke algebra.

  We can obtain Galois cohomology classes by composing $\mathcal{LE}^{[q,r]}$ with the $p$-adic \'etale realisation map, for a choice of prime $p$. Let $\Pi$ be a cuspidal, non-CAP automorphic representation of $\GSp_4 / \QQ$ whose Archimedean component is cohomological with respect to the algebraic representation of highest weight $(r_1, r_2)$. By projecting to the $\Pif^\vee$-isotypical component of \'etale cohomology, we obtain classes in the Galois cohomology of $V_{\Pi}^*(-q)$, where $V_{\Pi}$ is the 4-dimensional $p$-adic Galois representation associated to $\Pi$. This step depends on a choice of homomorphism of Galois representations $H^3_{\et}(Y_{G, \QQbar}, \cD_{\Qp})[\Pif^\vee] \to V_{\Pi}^*$ (a ``modular parametrisation'').

  Our first main result shows that after restricting to a character eigenspace, this entire construction factors through a one-dimensional quotient. Let us choose a pair of Dirichlet characters $\uchi = (\chi_1, \chi_2)$ whose product is the central character $\chi_\Pi$ of $\Pi$, and let $\cS(\Af^4; \uchi^{-1})$ denote the $\uchi^{-1}$-eigenspace for the action of $\widehat{\ZZ}^\times \times \widehat{\ZZ}^\times$ on $\cS(\Af^4; \QQ)$.

  \begin{introtheorem}[{\cref{thm:equivariance}}]\label{introtheorem:A}
   Assume $\Pi_\ell$ is generic for all finite primes $\ell$. Then there is a class
   \[ z_{\can}^{[\Pi, q, r]}(\uchi) \in H^1\left(\QQ, V_{\Pi}^*(-q)\right)\]
   with the following property: for all choices of $\uPhi \in \cS(\Af^4; \uchi^{-1})$, all $\xi \in \cH(G(\Af);\QQ)$, and all choices of modular parametrisation of $V_{\Pi}^*$, the image of $\LE^{[q,r]}(\uPhi \otimes \xi)$ in $H^1(\QQ, V_{\Pi}^*(-q))$ is a scalar multiple of $z_{\can}^{[\Pi, q, r]}(\uchi)$. Moreover, the scalar factor is given by an explicit product of local zeta-integrals.
  \end{introtheorem}

  This is an analogue of the result of \cite{harrisscholl01} in the setting of the Beilinson--Kato Euler system for $\GL_2$; compare also \cite[\S 6.2]{loeffler-ggp} for an analogous, but less complete, result for Hilbert modular surfaces. The proof of \cref*{introtheorem:A} relies on multiplicity-one results for smooth representations of $\GSp_4$ over local fields due to M\oe glin--Waldspurger and R\"osner--Weissauer.

  Our second result shows that the genericity assumption is essential:

  \begin{introtheorem}[{\cref{thm:nongeneric}}]\label{introtheorem:B}
   Assume that the central character of $\Pi$ is a square in the group of Dirichlet characters. Then, if there is a finite prime $\ell$ such that $\Pi_\ell$ is non-generic, the projection of $\LE^{[q,r]}(\uPhi \otimes \xi)$ to the $\Pif^\vee$-isotypic component is zero for all choices of $\uPhi$ and $\xi$.
  \end{introtheorem}

  (We expect Theorem B to hold without the assumption on the central character; but this would require generalising some results in the Gan--Gross--Prasad theory of branching laws for smooth representations from special orthogonal groups to $\operatorname{GSpin}$ groups, and this theory has not yet been fully developed.)

  For the goal of studying Galois representations, the vanishing of the Euler system for non-generic $\Pi$ is no loss: for any $\Pi$ satisfying our hypotheses, there will be a unique globally generic representation $\Pi_{\mathrm{gen}}$ in the same $L$-packet as $\Pi$, and $\Pi_{\mathrm{gen}}$ will have the same Galois representation as $\Pi$, so we may construct an Euler system for this Galois representation using $\Pi_{\mathrm{gen}}$ instead. We prefer to view \cref{introtheorem:B} as showing that there is no redundant ``extra choice'' in constructing the Euler system given by varying the representation in its $L$-packet, just as \cref{introtheorem:A} shows that there is no redundancy obtained by varying the test data.

  \begin{remark}
   In the case of Yoshida lifts, $V_{\Pi}^*$ is a direct sum of two 2-dimensional representations, with one summand accounting for the largest and smallest Hodge--Tate weights, and the other summand the two intermediate ones. In this setting, combining \cref{introtheorem:B} and Arthur's multiplicity formula implies that the construction of \cite{LSZ17} can only give nontrivial classes in the summand corresponding to the middle two Hodge--Tate weights, never in the other summand. This is consistent with the Beilinson--Bloch--Kato conjecture: any cohomology class arising from geometry must lie in the Bloch--Kato $H^1_{\mathrm{g}}$ subspace, and the Beilinson--Bloch--Kato conjecture predicts that the summand of $V_{\Pi}^*(-q)$ with the outermost Hodge--Tate weights has $H^1_{\mathrm{g}} = 0$.
  \end{remark}

  We also give a precise \emph{formulation}, and make some preliminary steps towards the proof, of a third and much more difficult theorem:

  \begin{introtheorem}[proved in \cite{LZ26}; see \cref{thm:mainthm} for precise statement]
   Assume $\Pi$ is generic, $\Pi$ and the $\chi_i$ are unramified at $p$, and $\Pi$ is Klingen-ordinary at $p$. Then, for suitable vectors $\eta_{\dR} \in \Fil^1 \DdR(\Qp, V_{\Pi})$, we have
   \[ \left\langle \eta_{\dR}, \log z^{[\Pi, q, r]}_{\can}(\uchi) \right\rangle =
   (\star) \cdot \cL_{p, \nu}(\Pi,\uchi;  -1-r_2 + q, r). \]
   where $(\star)$ is an explicit factor, $\log$ denotes the Bloch--Kato logarithm, and $\cL_{p, \nu}(\Pi,\uchi; -, -)$ is the 2-variable $p$-adic $L$-function constructed in Proposition 10.4 of \cite{LPSZ1}.
  \end{introtheorem}

  If we assume $r_1 > r_2$ then $\cL_{p, \nu}(\Pi,\uchi; -, -)$ factors as the product of the 1-variable $p$-adic $L$-functions of $\Pi$ and $\Pi \times \chi_2^{-1}$; and we can recognise the right-hand side as the product of a critical $L$-value of $\Pi \times \chi_2^{-1}$ (depending on $r$) and a value of the $p$-adic $L$-function of $\Pi$ outside its interpolation range (depending on $q$).

  The main part of the proof of this theorem is given in the paper \cite{LZ26} (recently accepted for publication in Cambridge J.~Math.) In this paper we give only the ``smooth-representation-theoretic'' preliminaries, involving the evaluation of the local zeta integrals appearing in Theorem A for various explicit choices of test data. This is used to show that Theorem C is equivalent to a more concrete, but considerably messier, statement involving pushforwards of $\GL_2 \times \GL_2$ Eisenstein classes, which is the form of the statement which we shall actually prove in the accompanying paper.
  
  \subsubsection*{Acknowledgements} We thank the anonymous referee for their careful reading of the article and perceptive comments.

 \subsection{Notation and conventions}

  \begin{itemize}

   \item We define $G = \GSp_4 \subset \GL_4$ as the similitude group of the symplectic matrix $\begin{smatrix} &&& 1\\ 	&&1&\\&-1&&\\-1&&&\end{smatrix}$.

   \item Denote by $P_{\Sieg}$, $P_{\Kl}$ and $B$ the Siegel, Klingen and Borel parabolic subgroups of $G$, given by
   \begin{align*}
    P_{\Sieg}&=\begin{smatrix} \star&\star&\star&\star\\ \star&\star&\star&\star\\&&\star&\star\\&&\star&\star\end{smatrix},&
    P_{\Kl} &= \begin{smatrix} \star&\star&\star&\star\\ &\star&\star&\star\\
    &\star&\star&\star\\ &&&\star\end{smatrix},& B = P_{\Sieg} \cap P_{\Kl}.
   \end{align*}
  Write $M_{\Sieg}$ and $M_{\Kl}$ for the standard (block-diagonal) Levi subgroups of $P_{\Sieg}$ and $P_{\Kl}$, and $T = M_{\Sieg} \cap M_{\Kl}$ for the diagonal maximal torus. We write $N_{\Sieg}, N_{\Kl}$ and $N$ for the respective unipotent radicals.
  %

   \item Let $H = \{ (h_1, h_2) \in \GL_2 \times \GL_2: \det(h_1) = \det(h_2)\}$, and let $\iota$ denote the embedding $H \into G$ given by
   \[ \left(\begin{pmatrix} a & b\\ c & d\end{pmatrix}, \begin{pmatrix} a'& b'\\ c' & d'\end{pmatrix}\right) \mapsto
   \begin{smatrix} a&&& b\\ & a' & b' & \\ & c' & d' & \\ c &&& d\end{smatrix}.\]

   \item For $M \in \ZZ_{\ge 1}$, we shall identify a Dirichlet character $\chi : (\ZZ / M)^\times \to \CC^\times$ with the  unique continuous character of $\AA^\times / \QQ^\times$ that is unramified outside $M$ and maps $\varpi_\ell$ to $\chi(\ell)$ for $\ell \nmid M$, where $\varpi_\ell$ is any uniformizer at $\ell$. Note that the restriction of this adelic $\chi$ to $\widehat{\ZZ}^\times \subset \AA^\times$ is the composite of the projection $\widehat{\ZZ}^\times \to (\ZZ / M)^\times$ with the \emph{inverse} of $\chi$.

   \item In a slight conflict with the previous notation, if $j \in \ZZ$, and $\chi$ is a Dirichlet character conductor $p^m$ for some $m$ (valued in some $p$-adic field $L$), we write ``$j + \chi$'' for the continuous character $\Zp^\times \to L$ given by $x \mapsto x^j \cdot \chi(x \bmod p^m)$.
%
%
  \end{itemize}

\section{Multiplicity-one results for local periods}
 \label{sect:zeta-appendix}

 In this section we give a careful statement of a purely local uniqueness result in smooth representation theory (Theorem \ref{thm:appendix2}), and a semi-local variant, which will be used in the following sections to study our Euler system classes.

 \subsection{Local theory: $L$-packets}

  In this section we let $F$ be a finite extension of $\Ql$ for some prime $\ell$, and $|\cdot|$ the norm on $F$, normalised such that $|\varpi| = \frac{1}{q}$, where $\varpi$ is a uniformiser and $q$ is the order of the residue field. We let $\psi$ be a non-trivial additive character $F \to \CC^\times$, which we can regard as a character of $N(F)$ via
  \[ \begin{smatrix} 1 & x & \star & \star \\ & 1 & y & \star\\ && 1 & -x \\ &&& \phantom{-}1 \end{smatrix} \mapsto \psi(x + y).\]
  Recall that for an irreducible smooth $G(F)$-representation $\rho$, we have $\dim \Hom_{N(F)}(\rho, \psi) \le 1$, and we say $\rho$ is \emph{generic} if equality holds, in which case there is a unique subspace $\cW(\rho) \subset \Ind_{N(F)}^{G(F)}(\psi)$ isomorphic to $\pi$ (the \emph{Whittaker model} of $\rho$).

  The set of isomorphism classes of irreducible smooth $G(F)$-representations is partitioned into \emph{$L$-packets}, which can be defined as the fibres of the local Langlands correspondence of \cite{gantakeda11}. Each $L$-packet has size either 1 or 2, and is naturally parameterised by the characters of the centraliser of the Langlands parameter associated with the packet. We say an $L$-packet is \emph{generic} if it contains a generic representation; in this case it contains precisely one such representation, corresponding to the trivial character of the component group. All non-singleton $L$-packets are generic. (See \cite[Main Theorem]{gantakeda11} for these statements.)

  \begin{proposition}
   If $P$ is an $L$-packet containing a tempered (or essentially tempered) representation, then all members of $P$ are tempered (resp.~essentially tempered), and moreover $P$ is a generic packet.
  \end{proposition}
  
  \begin{proof} This is not explicitly stated in \cite{gantakeda11}, but follows from the explicit description of the correspondence given in \emph{op.cit.}, as follows.
  
  Firstly, in \S 7 of \cite{gantakeda11} it is shown that the non-singleton $L$-packets have the form $\{ \theta(\tau_1 \boxtimes \tau_2), \theta(\tau_1^D \boxtimes \tau_2^D)\}$ where $\tau_i$ are discrete-series representations of $\GL_2(F)$ (with the same central character) and $\tau_i^D$ are their Jacquet-Langlands transfers to the non-split quaternion algebra over $F$. (See \emph{op.cit.} for the meaning of the notation $\theta(-)$.) Representations of this form are always essentially tempered, and are tempered iff their central character (which is the same for both members of the packet) is unitary. So if one member of an $L$-packet $P$ is tempered or essentially tempered, then all are. Moreover, any representation which is essentially tempered but not generic arises as  $\theta(\tau_1^D \boxtimes \tau_2^D)$, by Theorem 5.6(i) of \cite{gantakeda11}, and hence lies in a 2-element $L$-packet whose other member $\theta(\tau_1 \boxtimes \tau_2)$ is generic.\end{proof}

 \subsection{Local multiplicity one for irreducible representations}

  We will need the following important theorem:

  \begin{theorem}[Prasad, Emory--Takeda]\label{thm:prasadET}
   Let $\pi$ be an irreducible $G(F)$-representation belonging to a generic $L$-packet, and $\sigma_1$, $\sigma_2$ irreducible $\GL_2(F)$-representations. Then we have $\dim \Hom_{H(F)}(\pi \otimes (\sigma_1 \boxtimes \sigma_2), \CC) \le 1$.

   Moreover, if $P$ is a generic $L$-packet for $G(F)$, and the $\sigma_i$ are generic representations, there is at most one $\pi \in P$ for which $\Hom_{H(F)}(\pi \otimes (\sigma_1 \boxtimes \sigma_2), \CC) \ne 0$.
  \end{theorem}

  \begin{proof}
   We note that for any irreducible representations $\pi$ of $G(F)$ and $\sigma$ of $H(F)$, the space $\Hom_{H(F)}(\pi \otimes \sigma, \CC)$ has dimension $\le 1$, by the main theorem of \cite{emorytakeda21}. If either of the representations $\sigma_i$ is non-generic (hence one-dimensional), then $\sigma_1 \boxtimes \sigma_2$ remains irreducible when restricted to $H$, so the first assertion follows in this case.

   If both $\sigma_i$ are generic, then $(\sigma_1 \boxtimes \sigma_2)|_{H(F)}$ may be reducible, but it is always a direct sum of finitely many distinct irreducibles. Theorem 5 of \cite{prasad96} shows that in this case there is at most one $\pi \in P$, and at most one irreducible factor $\sigma$ of $\sigma_1 \boxtimes \sigma_2$, for which the Hom-space is non-zero; and the aforementioned theorem of Emory--Takeda shows that this unique non-zero Hom-space (if it exists) has dimension 1. So the result also follows in this case.
  \end{proof}

 \subsection{Zeta integrals}
  We fix a pair of characters $(\chi_1, \chi_2)$ of $F^\times$, with $\chi_1 \chi_2$ equal to the central character $\chi_{\pi}$ of $\pi$. Let $\pi$ be a generic irreducible smooth representation of $G(F)$, and $\cW(\pi)$ its Whittaker model with respect to $\psi$.

  \begin{notation}
   The notation $\cS(F^2)$ denotes the space of $\CC$-valued Schwartz functions (locally constant functions of compact support) on $F^2$. We write $\cS_0(F^2)$ for the subspace of functions vanishing at $(0, 0)$.
  \end{notation}

  Hence there is a natural right-translation action of $\GL_2(F) \times \GL_2(F)$, and in particular of $H(F)$, on the space $\cS(F^2 \times F^2) = \cS(F^2) \otimes \cS(F^2)$, preserving the subspaces $\cS_0 \otimes \cS$, $\cS \otimes \cS_0$ and $\cS_0 \otimes \cS_0$.

  In \cite[\S 8.2]{LPSZ1}, we defined a local zeta-integral $Z(w, \uPhi, s_1, s_2)$, for $\uPhi = \Phi_1 \otimes \Phi_2 \in \cS(F^2) \otimes \cS(F^2)$ and $w \in \cW(\pi)$; it is given by the meromorphic continuation of the integral
  \[ \int_{(Z_H N_H \backslash H)(F)} w(h) f^{\Phi_1}(h_1; \chi_1, s_1) W^{\Phi_2}(h_2; \chi_2, s_2)\, \mathrm{d}h\]
  where $f^{\Phi_1}(h_1; \chi_1, s_1)$ is a Godement--Siegel section (living in a principal-series $\GL_2$ representation) and $W^{\Phi_2}(h_2; \chi_2, s_2)$ is a $\GL_2$ Whittaker function (with respect to $\psi^{-1}$). See \emph{op.cit.} for exact definitions. In Theorem 8.8 of~\emph{op.cit.} we showed (as a consequence of the computations of \cite{roesnerweissauer17, roesnerweissauer18}) that the fractional ideal of $\CC[q^{\pm s_1 \pm s_2}]$ generated by the values of the zeta-integral, as $w$ and $\uPhi$ vary, is the principal ideal generated by $L(\pi \times \chi_2, s_1 - s_2 + \tfrac{1}{2}) L(\pi, s_1 + s_2 - \tfrac{1}{2})$. Accordingly, the quotient
  \begin{equation}
   \label{eq:ztilde-def}
   \wZ(w, \uPhi, s_1, s_2) \coloneqq \lim_{(\xi_1, \xi_2) \to (s_1, s_2)} \frac{Z(w, \uPhi, \xi_1, \xi_2)}{L(\pi \times \chi_2^{-1}, \xi_1 - \xi_2 + \tfrac{1}{2}) L(\pi, \xi_1 + \xi_2 - \tfrac{1}{2})}
  \end{equation}
  is a well-defined and non-zero map $\cW(\pi) \times \cS(F^2) \times \cS(F^2) \to \CC$, for every $(s_1, s_2) \in \CC^2$. 
  
  \begin{lemma}
   For any $h' \in H(F)$ we have
   \begin{subequations}
    \begin{equation}
    \label{eq:equivariance1}
    \wZ\left(\iota(h')\cdot w, h' \cdot\uPhi, s_1, s_2\right) = \left|\det h'\right|^{-(s_1 + s_2)} \wZ(w, \uPhi, s_1, s_2);
    \end{equation}
    and if $\underline{a} = \left(\stbt{a_1}{0}{0}{a_1}, \stbt{a_2}{0}{0}{a_2}\right)$, for $a_i \in F^\times$, then
    \begin{equation}
    \label{eq:equivariance2}
     \wZ(w, \underline{a}\cdot \uPhi, s_1, s_2) = \frac{\wZ(w, \uPhi, s_1, s_2)}{|a_1|^{2s_1} \chi_1(a_1)\, |a_2|^{2s_2} \chi_2(a_2)}.
    \end{equation}
   \end{subequations}
  \end{lemma}
  
  \begin{proof}
   From the formulae defining $f^{\Phi}$ and $W^{\Phi}$ (see \cite[\S 8.1]{LPSZ1}), we have 
   \[ f^{g' \cdot\Phi}(g; \chi, s) = |g'|^{-s} f^{\Phi}(g g', \chi, s), \qquad W^{g' \cdot\Phi}(h; \chi, s) = |g'|^{-s} W^{\Phi}(g g'; \chi, s)\]
   for all $g, g' \in \GL_2(F)$ and $\Phi \in \cS(F^2)$. So if the integral defining $Z(w, \uPhi, s_1, s_2)$ is convergent (which is true for $\operatorname{Re}(s_1) \gg 0$, for any given $s_2$), we obtain the relation 
   \[ Z\left(\iota(h') \cdot w, h' \cdot \uPhi, s_1, s_2\right) = \left|\det h'\right|^{-(s_1 + s_2)} Z(w, \uPhi, s_1, s_2)\]
   by substituting $hh'$ for $h$ in the integral; and by the uniqueness of meromorphic continuation this holds for all values of $s_1$ such that either side is defined. Thus the same equivariance property also holds for $\wZ$. The proof of the second formula is similar.
  \end{proof}
  \begin{theorem}\label{thm:appendix2}
   Let $(s_1, s_2) \in \CC^2$ with both real parts $\le 0$, and suppose that $\pi$ is tempered and $\chi_1$, $\chi_2$ are unitary. Then the space of linear functionals satisfying \eqref{eq:equivariance1} and \eqref{eq:equivariance2} is one-dimensional, and $\wZ$ is a basis of this space.
  \end{theorem}

  \begin{proof}
   By \cite[Proposition 3.3(a)]{loeffler-zeta1}, any linear functional satisfying \eqref{eq:equivariance2} must factor through a map from $\cS(F^2 \times F^2)$ to the principal-series representation
   \[ \sigma \coloneqq \Ind\left( |\cdot|^{\tfrac{1}{2}-s_1} \chi_{1}^{-1}, |\cdot|^{s_1-\tfrac{1}{2}}\right) \boxtimes \Ind\left( |\cdot|^{\tfrac{1}{2} - s_2} \chi_{2}^{-1}, |\cdot|^{s_2- \tfrac{1}{2}}\right),\]
   of $\GL_2(F) \times \GL_2(F)$; and if \eqref{eq:equivariance1} also holds, then this functional must define a non-zero element of
   \[ \Hom_{H(F)}\left( \pi \otimes \sigma, \CC\right). \]
   So it suffices to show that $\Hom_{H(F)}\left( \pi \otimes \sigma, \CC\right)$ has dimension $\le 1$. The vast majority of cases can be handled by one or other of the following two results:
   \begin{itemize}
    \item If $\sigma$ is irreducible as a representation of $\GL_2(F) \times \GL_2(F)$, then this is an instance of \cref{thm:prasadET}.

    \item If $\chi_{\pi}$ is a square in the group of characters of $F^\times$, then we can reduce to the case when $\pi$, $\sigma$ have trivial central characters. In this case, $\pi$ and $\sigma$ factor through $\operatorname{SO}(5, F)$ and $\operatorname{SO}(4, F)$ respectively, and the representation $\sigma$ is of the type considered in \cite[\S 1.3]{moeglinwaldspurger12}\footnote{Note that our $\pi$ is the $\sigma$ of \emph{op.cit.}, and our $\sigma$ is their $\sigma'$.}; so the Proposition \emph{loc.cit.} shows that in this case we have $\dim \Hom_{H(F)}\left( \pi \otimes \sigma, \CC\right) \le 1$ (whether or not $\sigma$ is reducible).
   \end{itemize}

   So it suffices to suppose that $\sigma$ is reducible as an $H$-representation, and its central character is not a square. The assumption $\operatorname{Re}(s_i) \le 0$ implies that $\sigma$ has a unique generic irreducible constituent $\sigma_0$, and this appears as a subrepresentation; it is the image of $\cS_0(F^2) \otimes \cS_0(F^2)$ in $\sigma$. So we can consider the restriction map
   \[ \Hom_{H(F)}\left( \pi \otimes \sigma, \CC\right) \longrightarrow \Hom_{H(F)}\left( \pi \otimes \sigma_0, \CC\right). \]
   The space $\Hom_{H(F)}\left( \pi \otimes \sigma_0, \CC\right)$ has dimension $\le 1$, by the aforementioned result of Emory--Takeda.

   The restriction of the zeta-integral to $\cS_0(F^2) \otimes \cS_0(F^2)$ is analysed in detail in \cite{loeffler-zeta1}, building on the works of R\"osner--Weissauer cited above. These results show that $\wZ$ vanishes on $\pi \otimes \sigma_0$ if and only if $s = s_1 + s_2 - \tfrac{1}{2}$ is a \emph{subregular pole} for the two-parameter zeta-integral, in the sense of \cite[Definition 4.8]{loeffler-zeta1} -- equivalently, there exists a choice of Bessel model for which it is a subregular pole for the Bessel zeta integral in the sense of \cite{roesnerweissauer18}. A case-by-case enumeration carried out in \S 5 of \cite{roesnerweissauer18} shows that if $s_1 + s_2 - \tfrac{1}{2}$ is not a subregular pole, then in fact $\Hom_{H(F)}\left(\pi \otimes \sigma/\sigma_0, \CC\right) = 0$, so the above restriction map is bijective. So we need to show that if $\chi_{\pi}$ is not a square, then $s_1 + s_2 - \tfrac{1}{2}$ cannot be a subregular pole.

   Every generic irreducible representation of $\GSp_4(F)$ is either supercuspidal, or one of the Sally--Tadic types \{Ia, $\dots$, XIa\} (classified in \cite{sallytadic93}; see Appendix A of \cite{robertsschmidt07} for summary tables). There is nothing to prove in the supercuspidal cases or for types VII, VIIIa, IXa, since the $L$-factors for these representations have no poles at all (and a subregular pole is \emph{a fortiori} a pole). For types I and X, if $\pi$ is tempered, then all poles of its $L$-factor have imaginary part 0, while $\operatorname{Re}(s_1 + s_2 - \tfrac{1}{2}) \le -\tfrac{1}{2}$, so $s_1 + s_2 - \tfrac{1}{2}$ cannot be a pole. Types IIa, IVa, Va, VIa and XIa cannot arise, since for such representations the central character is always a square. This leaves only type IIIa; but it is shown in \cite{roesnerweissauer18} that $L$-factors for type IIIa representations never have any subregular poles.
  \end{proof}

  As a by-product of the proof we also obtain the following corollaries:

  \begin{corollary}
   If $\pi$ is a tempered irreducible principal series representation, $\chi_i$ are unitary characters, and $(s_1, s_2)$ have real parts $\le 0$, then any non-zero trilinear form on $\cW(\pi) \otimes \cS(F^2) \otimes \cS(F^2)$ satisfying \eqref{eq:equivariance1} and \eqref{eq:equivariance2} has non-zero restriction to $\cW(\pi) \otimes \cS_0(F^2) \otimes \cS_0(F^2)$.\qed
  \end{corollary}

  \begin{corollary}\label{cor:nongeneric}
   If $\chi_{\pi}$ is a square, the $\chi_i$ are unitary, and $\pi$ is a tempered but non-generic representation, then we have $\Hom_{H(F)}\left( \pi \otimes \sigma, \CC\right) = 0$.
  \end{corollary}

  \begin{proof}
   This follows from a careful study of the arguments of \cite{moeglinwaldspurger12}. In \S 1.2 of \emph{op.cit.} a multiplicity $m(\pi, \sigma^\vee)$ is defined for smooth representations $\pi$ of $\operatorname{SO}_d$ and $\sigma$ of $\operatorname{SO}_{d'}$ respectively, where $d, d'$ are integers with $d - d'$ odd. Taking $d = 5$, $d' = 4$, and identifying $G / Z_G$ with $\operatorname{SO}_5$ and $H / Z_G$ with $\operatorname{SO}_4$, we have $m(\pi, \sigma^\vee) = \dim \Hom_{H(F)}(\pi \otimes \sigma, \CC) = \dim \Hom_{H(F)}(\pi, \sigma^\vee)$, so we need to show that $m(\pi, \sigma^\vee) = 0$.
   
   The proposition in \S1.3 of \cite{moeglinwaldspurger12} gives a formula for this multiplicity, assuming that one or both of $\pi, \sigma$ is parabolically induced from a parabolic with Levi factor $\GL_{N_1} \times \dots \times \GL_{N_t} \times \operatorname{SO}_r$ for some $r \ge 0$ (subject to a restriction on the inducing data which corresponds to our condition on $\operatorname{Re}(s_i)$). In our case this holds for $\sigma$ with $r = 0$; and the result is that $m(\pi, \sigma^\vee) = m(\pi, \mathrm{triv})$ where $\mathrm{triv}$ denotes the trivial representation of the trivial group $\operatorname{SO}_0$. In this case, the definition of $m(\pi, \mathrm{triv})$ is precisely the dimension of the space of Whittaker functionals on $\pi$, so it is 1 if $\pi$ is generic and zero otherwise.
  \end{proof}

  (We expect \cref{cor:nongeneric} to hold without the assumption on the central character, but we do not know a reference where this is treated in full generality.)

 \subsection{Adelic results}

  Let $\Af$ denote the ring of finite ad\`eles of $\QQ$, and let $\Pif$ be an irreducible admissible smooth representation of $G(\Af)$, with finite-order central character $\chi$. We assume that the local factors $\Pi_\ell$ are tempered for all $\ell$; and we take $\chi_1, \chi_2$ to be arbitrary Dirichlet characters whose product is $\chi$. We let $\psi$ be a continuous additive character of $\AA$ trivial on $\QQ$, and we suppose $\Pif$ has a Whittaker model with respect to $\psi$.

  We shall apply the results of the previous sections to the $G(\Ql)$-representation $\Pi_{\ell}$, for every $\ell$, taking $(s_1, s_2) = (-\tfrac{t_1}{2}, -\tfrac{t_2}{2})$ for integers $t_i \ge 0$. Tensoring together the corresponding local maps, we obtain a bilinear form
  \[ \wZ_{\mathrm{f}} = \prod_{\ell} \wZ_\ell: \cW(\Pif) \otimes \cS(\Af^2) \times \cS(\Af^2) \to \CC, \]
  satisfying the obvious semilocal analogues of \cref{eq:equivariance1,eq:equivariance2}. From the analysis above, this must factor via an $H(\Af)$-equivariant map
  \[ \cW(\Pif) \otimes \Sigma_{\mathrm{f}} \to \CC, \]
  where $\Sigma_{\mathrm{f}}$ is an admissible (but not generally finite-length) principal-series $H(\Af)$-representation given by the product of the local principal-series representations $\sigma$ above for each $\ell$. It follows readily from Theorem \ref{thm:appendix2} that the space of such $H(\Af)$-equivariant homomorphisms is 1-dimensional, and $\wZ$ is a basis vector.

  Now let $\cS_{(0)}(\Af^2 \times \Af^2)$ denote the space of functions vanishing identically on the subspace $(0, 0) \times \Af^2$ if $t_1 = 0$, and along $\Af^2 \times (0, 0)$ if $t_2 = 0$ (and along both subspaces if $t_1 = t_2 = 0$). Let $\Sigma_{\mathrm{f}}^{(0)}$ denote the image of this space in $\Sigma_{\mathrm{f}}$. We make similar definitions locally, so that
  \[ \Sigma_{\mathrm{f}} / \Sigma_{\mathrm{f}}^{(0)} = \sideset{}{'}\bigotimes_{\ell} \Sigma_{\ell} / \Sigma_\ell^{(0)}.\]
  (Note this quotient is very often zero; in fact, if neither of the $t_i$ is 0, or if neither of the $\chi_i$ is the trivial character, then $\Sigma_{\mathrm{f}}^{(0)} = \Sigma_{\mathrm{f}}$.)

  \begin{proposition}
   The restriction map
   \[ \Hom_{H(\Af)}\left(\Pif \times \Sigma_{\mathrm{f}}, \CC\right)
   \longrightarrow \Hom_{H(\Af)}\left(\Pif \times \Sigma_{\mathrm{f}}^{(0)}, \CC\right)\]
   is a bijection. In particular the space $\Hom_{H(\Af)}\left(\Pif \times \Sigma_{\mathrm{f}}^{(0)}, \CC\right)$ is one-dimensional and spanned by the image of $\wZ_{\mathrm{f}}$.
  \end{proposition}

  \begin{proof}
   We follow essentially the same argument as \cite[Proposition 2.2]{harrisscholl01}, which is an analogous statement for Kato's Euler system. We know that there exist many primes $\ell$ such that the local restriction map
   \[ \Hom_{H(\Af)}\left(\Pi_\ell \times \Sigma_{\ell}, \CC\right) \to \Hom_{H(\Af)}\left(\Pi_\ell \times \Sigma_{\ell}^{(0)}, \CC\right)\]
   is a bijection: it suffices to take any prime for which $L(\pi_\ell, s_1 + s_2 - \tfrac{1}{2})$ is not a subregular pole, and by the preceding proof, any unramified prime will suffice.) So, given any $H$-invariant $\mathfrak{z}: \Pif \times \Sigma_{\mathrm{f}}^{(0)} \to \CC$, we can restrict $\mathfrak{z}$ to $\Sigma_\ell^{(0)} \otimes \sideset{}{'}{\bigotimes_{q \ne \ell}} \Sigma_{q}$, and this linear functional will extend uniquely to a linear functional $\tilde{\mathfrak{z}}$ on $\Sigma_{\mathrm{f}}$. One checks that $\tilde{\mathfrak{z}}$ actually agrees with $\mathfrak{z}$ on the whole of $\Sigma_{\mathrm{f}}^{(0)}$.
%
  \end{proof}
%

 \subsection{Fields of definition}

  In our global applications, $\Pif$ will be the finite part of an automorphic representation whose Archimedean component is determined by a pair of integers $r_1 \ge r_2 \ge 0$; and the twist $\Pif' \coloneqq \Pif \otimes \|\cdot\|^{-(r_1 + r_2)}$ will be definable over a number field $E$. There is a canonical bijection between $\cW(\Pif)$ and $\cW(\Pif')$, so we can regard $\wZ$ as a linear functional on $\cW(\Pif') \otimes \cS(\Af^2 \times \Af^2, \CC)$. We shall take $s_i =-\tfrac{t_i}{2}$ for integers $t_i \ge 0$.

  Since $\Pif'$ is definable over $E$, the Whittaker model has a canonical $E$-structure: we have $\cW(\Pif') = \cW(\Pif')_E \otimes_E \CC$, where $\cW(\Pif')_E \subset \cW(\Pif')$ is the space of $E$-rational Whittaker functions as defined as in \cite[Definition 10.2]{LPSZ1}.

  \begin{proposition}
   If $t_1 + t_2 = r_1 + r_2 \bmod 2$, and $E$ contains the values of the $\chi_i$ and also an $N$-th root of unity where $N$ is the conductor of $\chi_2$, then $\wZ$ is the base-extension to $\CC$ of a linear functional
   \[ \cW(\Pif')_E \otimes \cS(\Af^2 \times \Af^2, E) \to E. \]
  \end{proposition}

  \begin{proof}
   An elementary check, using the definition of the zeta-integral and the $E$-structure on the Whittaker model, shows that if $w \in \cW(\Pif')_E$ and $\uPhi \in \cS(\Af^2 \times \Af^2, E)$, then the values of the zeta-integral lie in $E \QQ^{\mathrm{ab}}$ and transform under $\Gal(E \QQ^{\mathrm{ab}} / E)$ via
   \[ \wZ(w, \uPhi)^\sigma = \chi_2(\sigma)^{2} \wZ(w, \uPhi), \]
   where $\chi_2$ is regarded as a character of $\Gal(\QQ^{\mathrm{ab}} / \QQ)$ by composing with the mod $N$ cyclotomic character. (This is a mild generalisation of Proposition 8.11 of \cite{LPSZ1}.) Thus if $E$ contains an $N$-th root of unity, this action is trivial for all $\sigma \in \Gal(E \QQ^{\mathrm{ab}} / E)$ and hence $\wZ(w, \uPhi)$ descends to $E$.
  \end{proof}

  \begin{remark}
   We can dispense with the need to introduce $N$-th roots of unity if we renormalise $\wZ$ by a Gauss sum, but we are principally interested in the case $\chi_2 = 1$ anyway, so we shall not do this here.
  \end{remark}

\section{Euler systems for Siegel automorphic representations}\label{sect:ES}

 Here we briefly recall the Galois cohomology classes constructed in \cite{LSZ17}, and use the multiplicity-one results of the previous section to understand the dependence of these classes on the auxiliary data we have chosen.

 \subsection{Automorphic representations}
  \label{sect:arthur}

  We recall some properties of automorphic representations for $\GSp_4$, following \cite[\S 10.1]{LSZ17}. We fix a Dirichlet character $\chi$.

  \subsubsection*{Endoscopic classification} The automorphic representations contributing to the discrete spectrum $L^2_{\mathrm{disc}}(G, \chi)$ have been classified by Arthur and Gee--Ta\"ibi: see \cite{arthur04} for an overview, Conjecture 2.5.6 of \cite{geetaibi18} for a precise statement for arbitrary GSpin groups, and Theorem 7.4.1 for its proof in the case of $\operatorname{GSpin}_5 = \GSp_4$. They can be partitioned into global packets, one for each ``Arthur parameter'' $\tau$, which is a formal sum of $\chi$-self-dual cuspidal automorphic representations of $\GL_1$ and $\GL_2$ satisfying certain conditions. We shall be interested in the case of \emph{semisimple} parameters $\tau$, which are given by either
  \begin{itemize}
  \item a single $\chi$-selfdual cuspidal automorphic representation of $\GL_4$ (``general type'');
  \item the formal sum $\pi \boxplus \pi'$ of two distinct cuspidal automorphic representations of $\GL_2$, both of central character $\chi$ (``Yoshida type'').
  \end{itemize}

  The corresponding global packet $\widetilde{\Pi}_{\tau}$ is a product of local packets $\widetilde{\Pi}_{\tau_v}$ of $G(\QQ_v)$-representations (and these local packets are singletons for almost all $v$). In the general-type case, every element in the global packet is automorphic, and has multiplicity 1 in the discrete spectrum. 
  
  In the Yoshida case, only ``half'' of these representations are automorphic. More precisely, for each place $v$, a choice of representation $\Pi_v \in \widetilde{\Pi}_{\tau_v}$ determines a character $\xi_{\Pi_v}$ of a certain centralizer group $\cS_\tau$ isomorphic to $\ZZ/ 2\ZZ$ (notations as in \S 5 of \cite{arthur04}); and the product $\sideset{}{'}\bigotimes_v \Pi_v$ is automorphic iff $\prod_v \xi_v = 1$ (and has multiplicity 1 in this case). This is an instance of Arthur's multiplicity formula (Theorem 7.4.1 of \cite{geetaibi18}).
  
  In either case, each local packet contains a unique generic representation (whose associated character $\xi_{\Pi_v}$ is trivial); and the tensor product of these locally-generic representations is automorphic and globally generic (a \emph{a priori} stronger condition than being everywhere locally generic). See Remark 7.4.7 of \cite{geetaibi18}.

  \subsubsection*{Cohomological representations}

  Let $r_1 \ge r_2 \ge 0$ be integers. This determines a pair of discrete-series representations $\{\Pi_\infty^H, \Pi_\infty^W\}$ of $G(\mathbb{R})$ (forming an archimedean packet), with $\Pi_\infty^H$ holomorphic and $\Pi_\infty^W$ generic, both having central character $\operatorname{sign}^{(r_1 - r_2)}$; these are the ``discrete series representations of weight $(k_1, k_2)$'' in the sense of \cite[\S 10.1]{LSZ17} for $(k_1, k_2) = (r_1 + 3, r_2 + 3)$. We shall suppose that $\chi(-1) = (-1)^{r_1 - r_2}$, so that the central characters of these representations match the restriction of $\chi$ to $\mathbb{R}^\times$.

  \begin{remark}
   The automorphic representations with Archimedean component $\Pi_\infty^H$ are those generated by holomorphic Siegel cusp forms, possibly vector-valued, taking values in the representation $\Sym^{(r_1 - r_2)} \otimes \det^{(r_2 + 3)}$ of $U_2(\mathbb{R})$.
  \end{remark}

  We fix henceforth a semisimple Arthur parameter $\tau$, of central character $\chi$, such that $\widetilde{\Pi}_{\tau_\infty}$ is the above Archimedean packet determined by $(r_1, r_2)$; and a member $\Pi$ of the global packet $\widetilde{\Pi}_{\tau}$ which is automorphic. We write $\Pi'$ for the (non-unitary) twist $\Pi \otimes \|\cdot\|^{-(r_1 + r_2)}$. Then $\Pi'$ is cohomological with coefficients in the representation\footnote{The representation here denoted $V(r_1, r_2; r_1 + r_2)$ was called $V^{a, b}$ in \cite{LSZ17}, where $(a, b) = (r_2, r_1 - r_2)$. We apologise to the reader for the shift in notations.} $V(r_1, r_2; r_1 + r_2)$ (where, as in \cite{LPSZ1}, $V(r_1, r_2; c)$ for $c = r_1 + r_2 \bmod 2$ denotes the representation of highest weight $\diag(u, v, tv^{-1}, t u^{-1}) \mapsto u^{r_1} v^{r_2} t^{(c - r_1 - r_2)/2}$, so the central character is $x \mapsto x^c$). In particular, $\Pi'$ is \emph{C-algebraic} in the sense of \cite{buzzardgee14}, so we may choose a number field $E \subset \CC$ such that $\Pif'$ is definable over $E$.

 \subsection{Shimura varieties}\label{ss:Shimvarandcoeff}

  \begin{definition}
   For $U \subset G(\Af)$ a sufficiently small level, and $K$ a field of characteristic 0, let $Y_{G}(U)_{K}$ denote the base-extension to $K$ of the canonical $\QQ$-model of the level $U$ Shimura variety for $G$. We denote by $Y_{G, K}$ the pro-variety $\varprojlim_U Y_G(U)_{K}$.
  \end{definition}

  \begin{definition}
   For each algebraic representation $V$ of $G$, let $\cV$ denote the
   $G(\Af)$-equivariant relative Chow motive over $Y_{G, \QQ}$ associated to $V$ via Ancona's functor, as in \cite[\S 6.2]{LSZ17}.
  \end{definition}

  \begin{remark}
   Our conventions are such that the 4-dimensional defining representation $V(1, 0; 1)$ of $G$ corresponds to the relative motive $h^1(A)$, where $A$ is the universal abelian surface over $Y_{G, \QQ}$; and the 1-dimensional symplectic multiplier representation $V(0, 0; 2)$ maps to $\QQ(-1)[-1]$, where the square brackets $[-1]$ signify twisting the $G(\Af)$-action by the character $\|\cdot\|^{-1}$.
  \end{remark}

  Since $\cV$ is a relative Chow motive, motivic cohomology of $Y_G(U)$ with coefficients in $\cV$ is defined: we can choose $m,n$ such that $\cV$ is a direct factor of $W^{\otimes n}(m)$, where $W$ is the defining representation of $G$. Then $H^i_{\mot}(Y_G(U)_{\QQ}, \cV)$ is a direct summand of $H^{i+n}(A^n,\QQ(m))$, where $A$ is the universal abelian surface over $Y_{G}(U)_{\QQ}$. The equivariant structure implies that the cohomology has a Hecke action, so the direct limit over all levels $U$ is a $G(\Af)$-representation.

  We use the same symbol $\cV$ for the $p$-adic \'etale realisation of this motive, which is a locally constant \'etale sheaf of $\Qp$-vector spaces on $Y_G(U)_{\QQ}$, with a natural extension to the canonical integral model $Y_G(U)_{\ZZ[1/N]}$ if $U$ is unramified outside $N$.

 \subsection{Automorphic Galois representations}
  \label{ss:modularparam}

  Let $E$ be a number field over which $\Pif'$ is definable, as above, and $L$ the completion of $E$ at some prime above $p$. Taking $V = V(r_1, r_2; r_1 + r_2)$, the $\Pif'$-isotypical part of $H^3_{\et,c}(Y_G(U)_{\QQbar}, \cV) \otimes_{\Qp} L$ is isomorphic to the sum of $\dim\left( \Pif^U\right)$ copies of an $L$-linear Galois representation $V_\Pi$ (uniquely determined up to isomorphism). If $\Pi$ is of general type, then the semisimplification of $V_{\Pi}$ coincides with the 4-dimensional representation $\rho_{\Pi, p}$ associated to $\Pi$ as in \cite[\S\S 10.1-10.2]{LSZ17}. If $\Pi$ is of Yoshida type, $V_{\Pi}$ is a 2-dimensional direct summand of $\rho_{\Pi, p}$.

  \begin{remark}
   The representation $\rho_{\Pi, p}$ only depends on the Arthur parameter $\tau$ (so we should perhaps write it as $\rho_{\tau, p}$). Hence in the general-type case all members of the global packet $\widetilde{\Pi}_\tau$ have the same Galois representation (at least up to semisimplication, but it is expected that these representations are always irreducible, and this is known for $p$ large).

   However, in the Yoshida-type case, the representation $V_{\Pi}$ depends on the choice of $\Pi$ within the packet: for each possible $\Pif$, exactly one of $\Pif \otimes \Pi_\infty^H$ and $\Pif \otimes \Pi_\infty^W$ appears in the automorphic spectrum, and this determines which Galois representation we obtain as the $\Pif'$-part of \'etale cohomology. In this setting, the Arthur parameter $\tau$ is given by $\pi_f \boxplus \pi_g$ where $f$ and $g$ are holomorphic modular forms, of weights $r_1 + r_2 + 4$ and $r_1 - r_2 + 2$ respectively, and we have $\rho_{\tau, p} = \rho_{f, p} \oplus \rho_{g, p}(-1-r_2)$. Then $V_{\Pi}$ will be $\rho_{f, p}$ if $\Pi_\infty = \Pi_\infty^H$, and $\rho_{g, p}(-1-r_2)$ if $\Pi_\infty = \Pi_\infty^W$.

   In particular, if we choose $\Pi$ such that $\Pi_v$ is generic for all finite $v$, then the local characters $\xi_{\Pi_v}$ of $\ZZ / 2\ZZ$ (see \cref{sect:arthur} above) are trivial for $v \ne \infty$. Since $\Pi$ is automorphic, Arthur's multiplicity formula shows that $\xi_{\Pi_\infty}$ must be trivial as well, implying that $\Pi_\infty$ must be $\Pi_\infty^W$. So the Galois representation we obtain is $\rho_{g, p}(-1-r_2)$.
  \end{remark}

  It is convenient to fix a representative of this isomorphism class, as follows.

  \begin{definition}
   Let $\cM$ denote an arbitrary (but fixed) choice of $E$-model of $\Pif'$, and $\cM_L$ its base-extension to $L$. Then we define
   \[ V_{\Pi} \coloneqq \Hom_{L[G(\Af)]}\Big(\cM_L, H^3_{\et, c}( Y_{G, \QQbar}, \cV)_L \Big).\]
  \end{definition}

  The representation $V_{\Pi}$ is then an $L$-linear representation of $\Gal(\QQbar/\QQ)$, either 2-dimensional or 4-dimensional, which is a distinguished representative of the isomorphism class of representations above. With these definitions, we obtain a \emph{canonical} isomorphism of $G(\Af) \times \Gal(\QQbar/\QQ)$-representations
  \[
   \cM_L \otimes_L V_{\Pi} \xrightarrow{\ \cong\ } H^3_{\et,c}(Y_{G,\QQbar}, \cV)_L[\Pif'],
  \]
  where the notation $(\dots)[\Pif']$ denotes the $\Pif'$-isotypic part.

  \begin{remark}
   Later we shall suppose $\Pif$ is everywhere locally generic, and $\cM$ is the $E$-linear Whittaker model $\cW(\Pif')_E$. In this case, one can give a slightly more down-to-earth definition of $V_{\Pi}$ for generic $\Pi$ using the newvector theory of \cite{robertsschmidt07} and \cite{okazaki}. This allows us to choose a level $U(\Pi)$ such that $\cW(\Pif')^{U(\Pi)}$ is one-dimensional and has a canonical basis vector $w^0$ normalised such that $w^0(1) = 1$. Moreover, since non-generic tempered representations of $\GSp_4(\Ql)$ are not paramodular, no other representation in the global packet of $\Pi$ has a $U(\Pi)$-fixed vector. So evaluating at $w^0$ identifies $V_\Pi$ with the subspace of $H^3_{\et,c}(Y_G(U(\Pi))_{\QQbar}, \cV)_L$ on which the Hecke operators away from $\Sigma$ act via the system of eigenvalues assocated to $\Pif'$, where $\Sigma$ is any finite set of primes containing all those where $\Pi$ ramifies.
  \end{remark}

  There is a duality of $G(\Af) \times \Gal(\QQbar/\QQ)$-representations
  \[
   \big\langle\!\!\big\langle -, -\big\rangle\!\!\big\rangle_G:  \Big(H^3_{\et,c}(Y_{G,\QQbar}, \cV)\Big) \times \Big(H^3_{\et}(Y_{G,\QQbar}, \cV^\vee(3))\Big) \longrightarrow L
  \]
  given at level $U$ by $\vol(U) \cdot \langle -, -\rangle_{Y_G(U)_{\QQbar}}$, where $\langle-, - \rangle_{Y_G(U)_{\QQbar}}$ is the Poincar\'e duality pairing on the cohomology of $Y_G(U)_{\QQbar}$, and ``vol'' denotes volume with respect to the unramified Haar measure on $G(\Af)$ (so $G(\widehat{\ZZ})$ has volume 1).

  \begin{proposition}
   There is a canonical isomorphism between $\cM_L$ and the space
   \[ \Hom_{\Gal(\QQbar / \QQ)}\left(H^3_{\et}\left(Y_{G,\QQbar}, \cV^\vee(3)\right)_L[\Pif'^\vee], V_{\Pi}^*\right), \]
   i.e.~the space of ``modular parametrisations'' of the Galois representation $V_{\Pi}^*$, in the sense of \cite[\S 10.4]{LSZ17}. We write $\big\langle\!\!\big\langle m, -\big\rangle\!\!\big\rangle_G$ for the modular parametrisation corresponding to $m \in \cM_L$. \qed
  \end{proposition}

 \subsection{The Lemma--Eisenstein map}\label{ss:ESclass}

  Let $(q, r)$ be integers with $0 \le q \le r_2$ and $0 \le r \le r_1 - r_2$.

  \begin{notation}
   We let $\cS_{(0)}(\Af^2 \times \Af^2, \QQ)$ denote the space of $\QQ$-valued Schwartz functions on $\Af^2 \times \Af^2$, with $\GL_2(\Af) \times \GL_2(\Af)$ acting by right-translation, satisfying the following vanishing property: if $t_1 = 0$, then $\Phi( (0, 0) \times -)$ vanishes identically; and if $t_2 = 0$, then $\Phi(- \times (0, 0))$ vanishes identically.
  \end{notation}

  In \cite[\S 8.3]{LSZ17}, we defined a map, the \emph{Lemma--Eisenstein map}
  \[ \mathcal{LE}^{[q,r]}: \cS_{(0)}(\Af^2 \times \Af^2, \QQ) \otimes \cH(G(\Af)) \to H^4_{\mot}(Y_{G,\QQ}, \cV^\vee(3-q)) \]
  satisfying a certain equivariance property, where $\cH(-)$ denotes the Hecke algebra with $\QQ$-coefficients. (Note that this map depends on a choice of Haar measures on $G(\Af)$ and $H(\Af)$; we take the unramified Haar measures on both.) Denote by $\mathcal{LE}_{\et}^{[q, r]}$ the composite of this map with the \'etale realisation.\footnote{More precisely, since \'etale cohomology is not well-behaved over global fields (the absolute Galois group lacks good finiteness properties), we regard $\mathcal{LE}_{\et}^{[q, r]}$ as a map into the direct limit of \'etale cohomology of smooth $\ZZ[1/S]$-models of Shimura varieties $Y_{G}(U)_{\QQ}$, with both $S$ and $U$ varying, and $S$ assumed to contain $p$ and the primes of ramification of $U$. The terms in this direct limit are finite-dimensional over $\Qp$, and $\mathcal{LE}_{\et}^{[q, r]}$ is well-defined as a map into the direct limit, since any element of $\cS_{(0)}(\Af^2 \times \Af^2, \QQ) \otimes \cH(G(\Af))$ is unramified at all but finitely many primes.}

  \begin{note}
   The \'etale cohomology of $Y_{G, \QQ}$ (in the sense of the footnote) is related to Galois cohomology of \'etale cohomology over $\QQbar$ via the Hochschild--Serre spectral sequence. Since no representation in the $L$-packet of $\Pif'$ can contribute to cohomology in degrees other than 3, the generalised eigenspace in $H^4_{\et}(Y_{G,\QQ}, \cV^\vee(3-q))$ for the spherical Hecke eigensystem corresponding to $\Pif'^\vee$ will be contained in the homologically trivial classes (the kernel of the map to $H^0(\QQ, H^4_{\et}(Y_{G, \QQbar}, \dots))$), which is the  domain of the \'etale Abel--Jacobi map into $H^1(\QQ, H^3_{\et}(Y_{G, \QQbar}, \dots))$.

   Hence the natural projection map onto the $\Pif'^\vee$-eigenspace
   \[ \pr_{\Pif'^\vee} : H^1(\QQ, H^3_{\et}(Y_{G,\QQbar}, \cV^\vee(3-q))_L \to H^1\left(\QQ, H^3_{\et}(Y_{G,\QQbar}, \cV^\vee(3-q))_L[\Pif'^\vee]\right)\]
   lifts to a map
   \begin{equation}
    \label{eq:def_prPi}
    \AJ^{[\Pi, q]} : H^4_{\et}(Y_{G,\QQ}, \cV^\vee(3-q))_L \to H^1\left(\QQ, H^3_{\et}(Y_{G,\QQbar}, \cV^\vee(3-q))_L[\Pif'^\vee]\right),
   \end{equation}
   characterised as the unique Hecke-equivariant map agreeing with the $\Pif'^\vee$-projection of the \'etale Abel--Jacobi map on homologically trivial classes. We denote the composite $\AJ^{[\Pi, q]} \circ \mathcal{LE}_{\et}^{[q, r]}$ by
   \[ \mathcal{LE}_{\et}^{[\Pi, q, r]}:  \cS_{(0)}(\Af^2 \times \Af^2, L) \otimes_L \cH(G(\Af))_L \longrightarrow H^1\left(\QQ, H^3_{\et}(Y_{G,\QQbar}, \cV^\vee(3-q))_L[\Pif'^\vee]\right).\]
   (Cf.~\cite[\S 10.3]{LSZ17}.)
   \end{note}

 \subsection{Dual formulation}

  For our purposes, it is simpler to work with the bilinear form corresponding to $\LE_{\et}^{[\Pi, q, r]}$ under Frobenius reciprocity as in \cite[\S 3.9]{LSZ17}. Any $m \in \cM_L$ gives a homomorphism $H^3_{\et}(\dots)_L[\Pif'^\vee] \to V_{\Pi}^*$, and hence a map
  \[
   \big\langle\!\!\big\langle m, \mathcal{LE}^{[\Pi,q,r]}_{\et}(-) \big\rangle\!\!\big\rangle_G: \cS_{(0)}(\Af^2 \times \Af^2, L) \otimes_L \cH_L(G(\Af)) \to H^1(\QQ, V_{\Pi}^*(-q)).
  \]
    
  \begin{definition}
   Let
   \[ z^{[\Pi, q, r]} : \cM_L \times \cS_{(0)}(\Af^2 \times \Af^2, L) \to H^1(\QQ, V_{\Pi}^*(-q)) \otimes \|\det\|^{-q}\]
   be the $L$-bilinear map defined as follows: we let 
   \[ z^{[\Pi, q, r]}(m, \uPhi) = \big\langle\!\!\big\langle m, \mathcal{LE}^{[\Pi,q,r]}_{\et}(\uPhi \otimes 1_U) \big\rangle\!\!\big\rangle_G,\]
   where $1_U \in \cH(G(\Af))_L$ is the idempotent attached to any open compact subgroup $U$ of $G(\Af)$ which fixes $m$. 
  \end{definition}
  
  From the $(G \times H)(\Af)$-equivariance properties of $\mathcal{LE}^{[\Pi,q,r]}_{\et}$ (cf.~\cite[\S 8.2]{LSZ17}), it follows that $z^{[\Pi, q, r]}(m, \uPhi)$ does not depend on the choice of $U$; and it is $H(\Af)$-equivariant, for the diagonal action of $H(\Af)$ on $\cM_L \times \cS_{(0)}(\Af^2 \times \Af^2, L)$ and the trivial action on $H^1(\QQ, V_{\Pi}^*(-q))$. 
  
  
  Unravelling the notation, the values of $z^{[\Pi, q, r]}$ can be made explicit as follows:

  \begin{proposition}
   \label{prop:explicitZ}
   Suppose that $m \in \cM_L$. Then for any open compact $U \subset G(\Af)$ such that $U$ fixes $m$ and $V = U \cap H(\Af)$ fixes $\uPhi$, we have
   \[ z^{[\Pi, q, r]}(m, \uPhi) = \vol(V) \cdot \Big\langle m, \iota^{[t_1,t_2]}_{U, \star}\left(\Eis^{[t_1,t_2]}_{\et, \uPhi}\right)\Big\rangle_{Y_G(U)}. \]
   Here $\iota^{[t_1,t_2]}_{U, \star}$ denotes pushforward along $Y_H(V)_{\QQ} \to Y_G(U)_{\QQ}$, and $\Eis^{[t_1,t_2]}_{\et, \uPhi}$ denotes the \'etale realisation of the motivic Eisenstein class $ \Eis^{[t_1,t_2]}_{\Mot, \uPhi}$ (c.f. \cite[\S 4.1]{KLZ20}). \qed
  \end{proposition}

 \subsection{The universal Euler system class}
  \label{sect:zeta1}

  We shall now apply the local results of \S \ref{sect:zeta-appendix} to the Galois-cohomology-valued bilinear form $z^{[\Pi, q, r]}(m, \uPhi)$. We have already seen that this is $H$-equivariant if we let $H$ act by $\|\det\|^{-q}$ on the target.

  \begin{proposition}
   For $a_1, a_2 \in \QQ^\times$, we have
   \[ z^{[\Pi, q, r]}\Big(m, \stbt{a_1}{}{}{a_1} \Phi_1 \otimes \stbt{a_2}{}{}{a_2}\Phi_2\Big) = a_1^{-t_1}\cdot a_2^{-t_2}\cdot
      z^{[\Pi, q, r]}\left(m, \Phi_1 \otimes \Phi_2 \right).\]
  \end{proposition}

  \begin{proof}
   This follows from the definition of the $\GL_2$ Eisenstein symbol, which factors through the maximal quotient of $\cS_{(0)}(\Af^2)$ on which $\QQ^\times$ acts via $x \mapsto x^{-t}$; see \S 7.2 of \cite{LSZ17} for a much more precise statement.
  \end{proof}

  Let us now choose a pair of $E$-valued Dirichlet characters $\uchi = (\chi_1, \chi_2)$ with $\chi_1\chi_2 = \chi$. Let $\cS_{(0)}(\Af^2 \times \Af^2; \uchi^{-1})$ denote the $(\chi_1^{-1}, \chi_2^{-1})$-eigenspace for the action\footnote{See remarks above on how we interpret Dirichlet characters as characters of $\hat{\ZZ}^\times$, which is the opposite of the na\"ive identification.} of $\hat{\ZZ}^\times\times \hat{\ZZ}^\times$ on $\cS_{(0)}(\Af^2 \times \Af^2, E)$, and similarly without ${(0)}$.

  \begin{corollary}
   If $(-1)^{q + r}$ is not equal to the common value $(-1)^{r_1} \chi_1(-1) = (-1)^{r_2} \chi_2(-1)$, then  $z^{[\Pi, q, r]}(m, \uPhi)$ vanishes for all $\uPhi \in \cS_{(0)}(\Af^2 \times \Af^2, \uchi^{-1})$. If these signs are equal, then for all $m \in \cM_L$, $\uPhi \in \cS_{(0)}(\Af^2 \times \Af^2, \uchi^{-1})$ and $a_1, a_2 \in \Af^\times$, we have
   \[
    z^{[\Pi, q, r]}\Big(m, \stbt{a_1}{}{}{a_1} \Phi_1 \otimes \stbt{a_2}{}{}{a_2}\Phi_2\Big) =
    \frac{z^{[\Pi, q, r]}\left(m, \Phi_1 \otimes \Phi_2 \right)}{\|a_1\|^{-t_1} \chi_1(a_1) \cdot \|a_2\|^{-t_2} \chi_2(a_2)}.
   \]
  \end{corollary}

  \begin{proof}
   We have $\Af^\times = \QQ^\times \cdot \widehat{\ZZ}^\times$ and the intersection of these subgroups is $\{\pm 1\}$, so the result follows from the previous proposition.
  \end{proof}

  \begin{theorem}\label{thm:nongeneric}
   Suppose that $\chi$ is a square in the group of Dirichlet characters. Then, if there exists a finite prime $\ell$ at which $\Pi_\ell$ is not locally generic, the bilinear form $z^{[\Pi, q, r]}$ is zero; equivalently, the projection of $\mathcal{LE}_{\et}^{[q, r]}$ to the $\Pif'^\vee$-isotypical component is the zero map.
  \end{theorem}

  \begin{proof}
   This follows from \cref{cor:nongeneric}. Every $\uPhi$ is a finite linear combination of eigenvectors for $\widehat{\ZZ}^\times \times \widehat{\ZZ}^\times$, and the restriction of $z^{[\Pi, q, r]}$ to such an eigenspace is zero, by the corollary.
  \end{proof}

  So we may restrict henceforth to representations $\Pi$ such that $\Pif$ is everywhere locally generic, and choose $\cM = \cW(\Pif')_E$ to be the Whittaker model. For concreteness we take $\psi$ to be the additive character of $\AA / \QQ$ of conductor 1 which restricts to $x \mapsto e^{-2\pi i x}$ on $\mathbf{R}$. Hence there is a  Whittaker model $\cW(\Pif')$ of $\Pif'$ with respect to $\psi$. We denote this space by $\cW(\Pif')$, and $\cW(\Pif')_E$ the subspace of Whittaker functions which are \emph{defined over $E$} in the sense of \cite[Definition 10.2]{LPSZ1}. This gives a canonical model of $\Pif'$ as an $E$-linear representation (and hence a canonical model of its Galois representation, as above).

  \begin{theorem}
   \label{thm:equivariance}
   There exists a class
   \[ z^{[\Pi, q, r]}_{\can}(\uchi) \in H^1(\QQ, V_{\Pi}^*(-q)) \]
   such that for all $(w, \uPhi) \in \cW(\Pif')_L \times \cS_{(0)}(\Af^2 \times \Af^2, \uchi^{-1})_L$, we have
   \[ z^{[\Pi, q, r]}\left(w, \uPhi\right) = z^{[\Pi, q, r]}_{\can}(\uchi) \cdot \wZ(w, \uPhi).\]
  \end{theorem}

  \begin{proof}
   We have shown that the map $z^{[\Pi, q, r]}$ factors through the maximal quotient of $\cW(\Pif')_L \times \cS_{(0)}(\Af^2 \times \Af^2, \uchi^{-1})_L$ on which $H$ acts by $\|\det\|^{-q}$ and the centres of the two $\GL_2$'s act by the characters $\|\cdot\|^{t_i} \chi_i^{-1}$. By the results of \cref{sect:zeta-appendix}, this is one-dimensional and has a canonical basis dual to the bilinear form $\wZ$. Hence the result.
  \end{proof}

  Note that if $S$ is any finite set of primes containing $p$ and the primes of ramification of $\Pi$ and $\uchi$, the class $z^{[\Pi, q, r]}_{\can}(\uchi)$ is unramified outside $S$, i.e.~it lies in the subspace $H^1(\ZZ[1/S], V_{\Pi}^*(-q))$.

\section{Explicit formulae at unramified primes}
 \label{sect:localzetas}

 We now evaluate the linear functionals $\wZ(w, \Phi, s_1, s_2)$ of \eqref{eq:ztilde-def} explicitly, for some specific choices of the test data. These results are used as part of the proof of the explicit reciprocity law in \cite{LZ26}; but we present them here since the methods of proof match the remainder of this paper.

 We shall let the local field $F$ be $\Qp$, and we shall take for $\pi$ the local factor $\Pi_p$ of a globally generic cuspidal automorphic representation $\Pi$ which is unramified at $p$ and has central character equal to our fixed Dirichlet character $\chi$. For compatibility with our global results we will be wanting to take Whittaker models relative to the restriction to $\Qp$ of the additive character $\psi$ above (which satisfies $\psi(1/p^n) = \exp(2\pi i / p^n)$ for all $n \in \ZZ$), although our results are valid more generally for any unramified additive character.

 As above, $r_1 \ge r_2$ are the weights of the algebraic representation for which $\Pi$ is cohomological, and we write $(\alpha, \beta, \gamma, \delta)$ for the Hecke parameters of $\pi' = \pi \otimes |\cdot|^{-(r_1 + r_2)}$. Our conventions are such that $\alpha\delta = \beta\gamma = p^{r_1 + r_2 + 3} \chi(p)$, and the temperedness of $\pi$ is equi\-val\-ent to the condition that $\alpha, \beta, \gamma, \delta$ all have complex absolute value $p^{(r_1 + r_2 + 3)/2}$.

 In this section $\chi_1, \chi_2$ will denote smooth characters of $\Qp^\times$ such that $\chi_1 \chi_2 = \chi_{\pi}$ (the local factors at $p$ of our global characters above). As before, we take $(q, r)$ with $0 \le q \le r_2$ and $0 \le r \le r_1 - r_2$, and consider the zeta integral at $(s_1, s_2) = (-\tfrac{t_1}{2}, -\tfrac{t_2}{2})$, where $(t_1, t_2) = (r_1 - q - r, r_2 - q + r)$; the two $L$-factors whose product is the GCD of values of the zeta-integral are therefore given by
 \begin{align*}
  L(\Pi, s_1 + s_2 - \tfrac{1}{2}) &= 
  \left[\left(1 - \tfrac{\alpha}{p^{q+1}}\right) \dots  \left(1 - \tfrac{\delta}{p^{q+1}}\right)\right]^{-1}, \\
  L(\Pi \times \chi_2^{-1}, s_1 - s_2 + \tfrac{1}{2}) &= 
  \left[\left(1 - \tfrac{\alpha}{p^{(r + r_2 + 2)}\chi_2(p)}\right) \dots  \left(1 - \tfrac{\delta}{p^{(r + r_2 + 2)}\chi_2(p)}\right)\right]^{-1}
 \end{align*}

 \subsection{Bases of eigenspaces at parahoric levels}

  Let $w^{\sph}$ be the spherical Whittaker function of $\pi$, normalised such that $w^{\sph}(1) = 1$ (which is always possible). We are interested in describing Hecke operators which will map $w^{\sph}$ to normalised generators of the eigenspaces at the various parahoric levels, where ``normalised'' again means these Whittaker functions take the value 1 at the identity.

  \begin{notation}[Congruence subgroups]
   We write $\Sieg(p^r)$ for the depth $r$ Siegel congruence subgroup (the preimage in $G(\Zp)$ of $P_{\Sieg}(\ZZ/p^r)$), and similarly $\Kl(p^r)$. We write $\Iw(p^r)$ (for ``Iwahori'') to denote the preimage of the mod $p^r$ Borel subgroup.
  \end{notation}

  We are primarily interested in the parahoric subgroups $\Sieg(p)$, $\Kl(p)$ and $\Iw(p)$. Note that the $\Sieg(p)$ and $\Kl(p)$-invariants of $\pi$ are each four-dimensional, and the $\Iw(p)$-invariants 8-dimensional.

  \begin{notation}[Hecke operators]
   Write $U_{1, \Sieg}$ for the operator $p^{(r_1 + r_2)/2} [\Sieg(p) \diag(p, p, 1, 1) \Sieg(p)]$ on $\pi^{\Sieg(p)}$, and $U_{2,\Kl}$ for $p^{r_1} [\Kl(p) \diag(p^2,p,p,1) \Kl(p)]$ acting on $\pi^{\Kl(p)}$. We define two operators $U_{1, \Iw}$ and $U_{2, \Iw}$ acting on the $\Iw(p)$-invariants similiarly.
  \end{notation}

  The normalising factors are chosen so that the eigenvalues of these operators are $p$-adically integral -- for $U_{1, \Sieg}$ or $U_{1, \Iw}$ these are $\{\alpha, \beta, \gamma, \delta\}$, and for $U_{2, \Kl}$ or $U_{2, \Iw}$ they are $\{ \tfrac{\alpha\beta}{p^{r_2 + 1}}, \dots\}$.

  \begin{lemma}
   \label{lemma:eigenvectors} \
   \begin{enumerate}
    \item The normalised generator of the $U_{1, \Sieg}=\alpha$ eigenspace of $\pi^{\Sieg(p)}$ is given by
    \[ w^{\Sieg}_{\alpha} = \left(1 - \frac{\beta}{U_{1, \Sieg}}\middle)\middle(1 - \frac{\gamma}{U_{1, \Sieg}}\middle)\middle(1 - \frac{\delta}{U_{1, \Sieg}}\right) w^{\sph}.\]
    \item If $\alpha + \gamma \ne 0$, the normalised generator of the $U_{2, \Kl}=\tfrac{\alpha\beta}{p^{r_2 + 1}}$ eigenspace of $\pi^{\Kl(p)}$ is given by
    \[ w^{\Kl}_{\alpha\beta} = \frac{1}{(1 + \tfrac{\gamma}{\alpha})} \left(1 - \frac{\beta\gamma}{p^{r_2 + 1}U_{2, \Kl}}\middle)\middle(1 - \frac{\alpha\gamma}{p^{r_2 + 1}U_{2, \Kl}}\middle)\middle(1 - \frac{\beta\delta}{p^{r_2 + 1}U_{2, \Kl}}\right)w^{\sph}. \]
    \item The normalised generator of the $(U_{1, \Iw} = \alpha, U_{2, \Iw} = \tfrac{\alpha\beta}{p^{r_2 + 1}})$ eigenspace in $\pi^{\Iw(p)}$ is given by
    \[ w^{\Iw}_{\alpha, \beta}
    = \left(1 - \frac{\alpha\gamma}{p^{r_2 + 1}U_{2, \Iw}}\right) w^{\Sieg}_{\alpha}
    = \left(1 - \frac{\beta}{U_{1, \Iw}}\right) w^{\Kl}_{\alpha\beta}. \]
   \end{enumerate}
  \end{lemma}

  In each case, it is obvious that the given vector lies in the relevant eigenspace, and the content of the lemma is that it takes the value 1 at the identity. This follows by explicit computations from the Casselman--Shalika formula giving the values of $w_{\sph}$ on any diagonal element; an explicit form of this formula in the $\GSp_4$ case can be found as Equation 7.3 in \cite{robertsschmidt07}.

  \begin{proposition}
   If $\alpha + \gamma \ne 0$, then the image of $w^{\Kl}_{\alpha\beta}$ under the trace map $v \mapsto \sum_{\gamma \in G(\Zp) / \Kl(p)} \gamma \cdot v$ is given by
   \[
    \operatorname{Tr}\left(w^{\Kl}_{\alpha\beta}\right)=
    p^3 \left( 1 - \frac{\gamma}{p\beta}\right) \left( 1 - \frac{\delta}{p\alpha}\right) \left(1 - \frac{\delta}{p\beta}\right)w^{\sph}.
   \]
  \end{proposition}
  \begin{proof}
   This follows from an extremely tedious explicit computation; rather than the Whittaker model, one fixes an ordering of the Hecke parameters, giving a choice of model of $\pi$ as an induction from the Borel subgroup. The Klingen-invariants then have an explicit basis given by coset representatives for $B(\Zp) \backslash G(\Zp) / \Kl(p)$. A lengthy double-coset computation gives the matrix of $U_{2, \Kl}$ in this basis; and the trace map in this basis is explicit, so the result follows from a routine computation.
  \end{proof}

  \begin{remark}
   A formula for $\operatorname{Tr}\left(w^{\Kl}_{\alpha\beta}\right)$ is stated without proof in \cite{genestiertilouine05}. However, their formula differs from ours, having terms of the shape $\left( 1 - \frac{\gamma}{\beta}\right)$ rather than $\left( 1 - \frac{\gamma}{p\beta}\right)$. We believe the formula stated above to be the correct one.
  \end{remark}

  \subsection{Dual parahoric eigenvectors} To link up with the zeta-integral computations of \cite{LPSZ1} we also needed to consider eigenvectors for the ``transpose'' Hecke operator $U_{2, \Kl}' = p^{r_1}\diag(1, p, p, p^2)$ at Klingen level. Given the above computations, it seems natural to consider the vector
  \begin{equation}
   \label{eq:wklprime}
   w^{\Kl\prime}_{\alpha\beta} \coloneqq \frac{1}{(1 + \tfrac{\gamma}{\alpha})} \left(1 - \frac{\beta\gamma}{p^{r_2 + 1}U'_{2, \Kl}}\middle)\middle(1 - \frac{\alpha\gamma}{p^{r_2 + 1}U'_{2, \Kl}}\middle)\middle(1 - \frac{\beta\delta}{p^{r_2 + 1}U'_{2, \Kl}}\right)w^{\sph}.
  \end{equation}
  By dualizing the previous computation, we see that $\operatorname{Tr}\left(w^{\Kl, \prime}_{\alpha\beta}\right) = \operatorname{Tr}\left(w^{\Kl}_{\alpha\beta}\right)$.

  We briefly summarize how this relates to the computations of \emph{op.cit.}. Let $w = r_1 + r_2 + 3$, and let $\Lambda$ be the unramified character of $T(\Qp)$ given by $\chi_1 \times \chi_2 \rtimes \rho$ in the notation of \cite{robertsschmidt07} \S 2.2, where $\rho(p) = p^{-w/2} \alpha$, $\chi_1(p) = \gamma/\alpha$, $\chi_2(p) =\beta/\alpha$. Then $\Ind_{B(\Qp)}^{G(\Qp)}(\Lambda)$ gives an explicit model of $\pi$; and identifying the Klingen Levi $M_{\Kl}$ with $\GL_2 \times \GL_1$ as in \cite[\S 8.4]{LPSZ1}, we can thefore write $\pi = \Ind_{P_{\Kl}(\Qp)}^{G(\Qp)}(\tau \boxtimes \theta)$, where $\theta = \chi_1$ and $\tau$ is the unramified principal series $\rho\chi_2 \times \rho$ of $\GL_2$ (with normalised Satake parameters $\{ p^{-w/2} \alpha, p^{-w/2} \beta\}$).

  In \emph{op.cit.} we considered the diagram of maps
  \begin{equation}
   \label{eq:Kldiagram}
   \begin{tikzcd}
   \pi \rar{\cong} \dar & \cW(\pi) \\
   \tau \rar{\cong} & \cW(\tau)
   \end{tikzcd}
  \end{equation}
  Here the horizontal arrows are the canonical intertwining maps from the induced representations to their Whittaker models, and the vertical arrow is given by restriction of functions in the induced representation from $G(\Qp)$ to $\GL_2(\Qp) \subset M_{\Kl}(\Qp)$ (note that this is only $\GL_2(\Qp)$-equivariant up to a twist by a power of $|\det|$).

  In \emph{op.cit.} we considered a vector $\phi_1 \in \pi^{\Kl(p)}$, characterized by the property of being supported on $B(\Qp) \cdot \Kl(p)$ and taking the value $p^3$ at the identity. This maps to $p^{3} \xi_{\sph} \in \tau$, where $\xi_{\sph}$ is the  spherical function of $\tau$ satisfying $\xi_{\sph}(1) = 1$. The map $\tau \to \cW(\tau)$ is given explicitly by $\xi \mapsto W_\xi$, where
  \[ W_{\xi}(g) = \int_{\Qp} \xi\left( \stbt{0}{-1}{1}0 \stbt{1}{x}{0}{1} g\right) \psi(-x)\, \mathrm{d}x. \]
  The Casselman--Shalika formula for $\GL_2(\Qp)$, cf.~\cite[Eq.~(6.11)]{bump97}, shows that $W_{\xi_{\sph}}(1) = \left(1 - \frac{\beta}{p\alpha}\right)$. Thus the image of $\phi_1$ in $\cW(\tau)$ is the spherical Whittaker function with value $p^3  \left(1 - \frac{\beta}{p\alpha}\right)$ at the identity. 

  \begin{lemma}
   The image of $\phi_1$ in $\cW(\pi)$ is $(1 - \tfrac{\beta}{p\alpha}) w^{\Kl\prime}_{\alpha\beta}$.
  \end{lemma}

  \begin{proof}
   Since both vectors are $U_{2, \Kl}'$-eigenvectors with the same eigenvalue, it suffices to check that they both have the same trace down to spherical level. By construction $\phi_1$ has trace $p^3 \phi_{\sph}$ where $\phi_{\sph}$ is the normalised spherical function of $\pi$. From the Casselman--Shalika formula for $\GSp_4$ \cite[Theorem 5.4]{CS80}, the image of $\phi_{\sph}$ in $\cW(\pi)$ has value at the identity equal to the quantity $\zeta(\chi)$ of \emph{op.cit.}, given explicitly by 
   \[ \left( 1 - \frac{\beta}{p\alpha}\right)\left( 1 - \frac{\gamma}{p\beta}\right) \left( 1 - \frac{\delta}{p\alpha}\right) \left(1 - \frac{\delta}{p\beta}\right)\]
   Hence the image of $p^3 \phi_{\sph}$ in $\cW(\pi)$ is $p^3  \left( 1 - \frac{\beta}{p\alpha}\right)\left( 1 - \frac{\gamma}{p\beta}\right) \left( 1 - \frac{\delta}{p\alpha}\right) \left(1 - \frac{\delta}{p\beta}\right) w_{\sph}$. This agrees with the formula we have computed above for $\operatorname{Tr}\left(w^{\Kl, \prime}_{\alpha\beta}\right)$.
  \end{proof}

  \begin{remark}
   \label{rem:correctWxi}
   Equivalently, $w^{\Kl, \prime}_{\alpha\beta}$ is the unique basis of the $U'_{2, \Kl}$-eigenspace in $\cW(\pi)^{\Kl(p)}$ which maps to the normalised spherical vector of $\cW(\tau)$ via the maps of diagram \eqref{eq:Kldiagram}.
  \end{remark}

 \subsection{Particular values of the zeta integral}

  \subsubsection{Spherical test data}

   We let $\Phi^{\sph} = \ch(\Zp^2 \times \Zp^2)$, and we let $w^{\sph}$ be the normalised spherical Whittaker function as above. Then, as we have already noted, we have
   \[ \wZ(w^{\sph}, \Phi^{\sph}) = 1. \]

  \subsubsection{Siegel parahoric test data}

   We choose a Hecke parameter $\alpha$ and consider the vector $w^{\Sieg}_{\alpha} \in \cW(\pi)^{\Sieg(p)}$ of \cref{lemma:eigenvectors}; and we let $\uPhi^{\Sieg} = \ch\left( (p\Zp \times \Zp^\times)^2\right)$. We shall compute $\wZ(w^{\Sieg}_{\alpha}, \uPhi^{\Sieg})$ using Bessel models.

   As in Proposition 8.4 of \cite{LPSZ1}, renormalising Novodvorsky's Whittaker integral by the factor $L(\pi \otimes \chi_2^{-1}, s)^{-1}$ and evaluating at $s = s_1 - s_2 + \tfrac{1}{2}$ gives a canonical nonzero intertwining map from the Whittaker model $\cW(\pi)$ to the Bessel model $\cB(\pi)$, with respect to a specific character of the Bessel subgroup depending on the $\chi_i$ and $s_i$; and this map sends the normalised spherical Whittaker function to the normalised spherical Bessel function.

   If $B_w \in \cB(\pi)$ is the image of $w \in \cW(\pi)$, and $\xi \in \CC$, we define
   \[
    z_{w, \xi}(h) \coloneqq \int_{\Qp^\times} B_w\left( \dfour{x}{x}{1}{1}h \right) |x|^{\xi + s_1 + s_2- 2}\, d^\times x,
   \]
   which is a meromorphic section of a 1-parameter family of principal series representations of $H$, dual to the one containing the Godement--Siegel sections $f^{\uPhi}(-; \uchi, s_1, s_2) = f^{\Phi_1}(-; \chi_1, s_1) \boxtimes f^{\Phi_2}(-; \chi_2, s_2)$ (cf.~\S 8.1 of \cite{LPSZ1}).
   \begin{proposition}
    \label{prop:besselZtilde}
    We have
    \[
     \wZ(w, \Phi, s_1, s_2) = \lim_{\xi \to 0} \tfrac{1}{L(\pi, \xi + s_1 + s_2 - \tfrac{1}{2})} \left\langle z_{w, \alpha}, f^{\uPhi}(s_1 + \tfrac{\xi}{2}, s_2 + \tfrac{\xi}{2}) \right\rangle.
    \]
    In particular, the limit on the right exists for all $(w, \Phi)$, and is non-zero for some $(w, \Phi)$.
   \end{proposition}

   \begin{proof}
    This is a restatement of \cite[Proposition 8.4]{LPSZ1}.
   \end{proof}

   \begin{proposition}
    \label{prop:siegelzeta}
    We have
    \[ \wZ(w^{\mathrm{\Sieg}}_\alpha, \Phi^{\Sieg}) =  \frac{1}{(p+1)^2} \left(1 - \frac{\beta}{p^{1+q}}\right)\left(1 - \frac{\gamma}{p^{1 + q}}\right)\left(1 - \frac{\delta}{p^{1 + q}}\right)  \left(1 - \frac{\delta}{ p^{r_2 + 2 + r} \chi_2(p)}\right) \left(1 - \frac{\chi_2(p) p^{r_2 + 1 + r} }{\alpha}\right).\]
   \end{proposition}

   \begin{proof}
    We use the Bessel-model description of $\wZ(w, \Phi, s_1, s_2)$ given in \cref{prop:besselZtilde}. Note that the choice of Bessel model used depends on the value of $r$, but is independent of $q$. The Schwartz function $\Phi^{\Sieg}$ is chosen so that $f^{\Phi^{\Sieg}}$ is supported on the coset $B_H(\Qp) \cdot \Iw_H$, where $\Iw_H = \Sieg(p) \cap H(\Qp)$ is the upper-triangular Iwahori subgroup of $H$, and its value at the identity is 1. The volume of this double coset (for the unramified Haar measure) is $\tfrac{1}{(p+1)}$; so for any $\Sieg(p)$-invariant (or just $\Iw_H$-invariant) $w$, we have
    \begin{align*}
     \wZ(w, \Phi^{\Sieg}) &=
    \lim_{\xi \to 0} \frac{1}{L(\pi, s_1 + s_2 - \tfrac{1}{2} + \xi)}\int_{B_H(\Qp) \backslash H(\Qp)} z_{w, \xi}(h) f^\Phi(h, s_1 + \tfrac{\xi}{2}, s_2 + \tfrac{\xi}{2})\ \mathrm{d}h \\
    &= \frac{z_{w, 0}(1)}{(p+1)^2 L(\pi, s_1 + s_2 - \tfrac{1}{2})}.
    \end{align*}
    (Note that the denominator is finite, since $\pi$ is tempered and $\Re(s_i) \le 0$.) Since $w$ is by assumption $\Sieg(p)$-invariant, we have $z_{w,0}(1) = F_w(p^{-1-q})$, where $F_w(X)$ is the rational function
    \[ F_w(X) = \sum_{n \in \ZZ} p^{n(r_1+r_2+ 6)/2}B_w\left(\dfour {p^n} {p^n} 1 1 \right) X^n. \]

    It is easy to see that for any $\Sieg(p)$-invariant $w$ we have
    \[ F_w(X) = F_w(0) + X F_{U_{1, \Sieg} \cdot w}(X), \]
    so in particular $F_{w_\alpha^{\Sieg}}(X)$ is a constant multiple of $1/(1 - \alpha X)$. We can determine the constant by comparing with the spherical Whittaker vector $w^{\sph}$: by Proposition 3.5.6(b) of \cite{LSZ17}, we have
    \[ F_{w^{\sph}}(X) =
     \frac{\left(1 - \chi_1(p) p^{r_1 + 1 - r} X\right) \left(1 - \chi_2(p) p^{r_2 + 1 + r} X\right)}{(1 - \alpha X)(1 - \beta X)(1 - \gamma X)(1 - \delta X) }.
    \]
    Supposing first that $(\alpha, \beta, \gamma, \delta)$ are distinct, we can express this in partial fraction form,
    \[
     F_{w^{\sph}}(X) = \frac{c_\alpha}{1 -\alpha X} + \dots + \frac{c_\delta}{1 - \delta X},\tag{\dag} 
    \]
    where $c_\alpha, \dots, c_\delta$ are constants. We can compute $c_\alpha$ by taking leading terms when $X = 1/\alpha$:
    \[
     c_\alpha =
     \frac{
      \left(1 - \tfrac{\chi_1(p) p^{r_1 + 1 - r}}{\alpha}\right)
      \left(1 - \frac{\chi_2(p) p^{r_2 + 1 + r} }{\alpha}\right)
      }{
      (1 - \tfrac{\beta}{\alpha})(1 - \tfrac{\gamma}{\alpha})(1 - \tfrac{\delta}{\alpha})
      }.
    \]

    On the other hand, if $P_\alpha$ denotes the operator $\frac{1}{\alpha^3}\left(U_{1, \Sieg} - \beta\right)  \left(U_{1, \Sieg} - \gamma\right)  \left(U_{1, \Sieg} - \delta\right)$, which maps $w_{\sph}$ to $w_\alpha^{\Sieg}$, and similarly $P_\beta$ etc, then we have the equality of endomorphisms
    \[ \id = \frac{P_\alpha}{(1 - \tfrac{\beta}{\alpha})(1 - \tfrac{\gamma}{\alpha})(1 - \tfrac{\delta}{\alpha})} + \dots + \frac{P_\delta}{(1 - \tfrac{\alpha}{\delta})(1 - \tfrac{\beta}{\delta})(1 - \tfrac{\gamma}{\delta})}, \]
    and hence
    \[ w^{\sph} = \frac{w_{\alpha}^{\Sieg}}{(1 - \tfrac{\beta}{\alpha})(1 - \tfrac{\gamma}{\alpha})(1 - \tfrac{\delta}{\alpha})} + \dots +  \frac{w_{\delta}^{\Sieg}}{(1 - \tfrac{\alpha}{\delta})(1 - \tfrac{\beta}{\delta})(1 - \tfrac{\gamma}{\delta})}.\]
    Thus we must have
    \[ F_{w^{\sph}}(X) = \frac{F_{w_{\alpha}^{\Sieg}}(X)}{(1 - \tfrac{\beta}{\alpha})(1 - \tfrac{\gamma}{\alpha})(1 - \tfrac{\delta}{\alpha})} + \dots,\]
    where $F_{w^{\Sieg}_{\alpha}}$ is a constant multiple of $\frac{1}{(1 - \alpha X)}$ and similarly for the other terms. Comparing this with ($\dag$), we conclude that 
    \[ \frac{c_\alpha}{1 -\alpha X} = \frac{1}{(1 - \tfrac{\beta}{\alpha})(1 - \tfrac{\gamma}{\alpha})(1 - \tfrac{\delta}{\alpha})} F_{w^{\Sieg}_{\alpha}}(X), \]
    or in other words
    \[
     F_{w^{\Sieg}_{\alpha}}(X) = \frac{ \left(1 - \tfrac{\chi_1(p) p^{r_1 + 1 - r}}{\alpha}\right)
     \left(1 - \frac{\chi_2(p) p^{r_2 + 1 + r} }{\alpha}\right)}{(1 - \alpha X)}.
    \]
    We have so far assumed that the Hecke parameters are distinct, but since the coefficients of $F_{w_{\alpha}^{\Sieg}}$ are rational functions of the Hecke parameters, the above formula holds for any irreducible unramified principal series by analytic continuation.
    Substituting this into the formula $\wZ(w_\alpha^{\Sieg}, \Phi^{\Sieg}) =\frac{1}{(p+1)^2}L(\Pi_p, s_1 + s_2 - \tfrac{1}{2})^{-1} \cdot F_{w^{\Sieg}_{\alpha}}\left(p^{-1-q}\right)$ gives the result.
   \end{proof}

  \subsubsection{Klingen test data: the general formula}

   We now consider the case of test vectors of Klingen parahoric level. This computation is largely worked out in \cite{LPSZ1} \S 8.4, but without making explicit the normalisation of the vector $w \in \cW(\Pi)$ used, so we shall tease out this detail.

   \begin{notation}
    We let $u_{\Kl}$ denote any element of $\GSp_4(\Zp)$ whose first column is $\begin{smatrix}1 \\ 1 \\ 0 \\ 0\end{smatrix}$.
   \end{notation}

   We shall assume that our Hecke parameters are ordered as $(\alpha, \beta, \gamma, \delta)$ with the common ratio $\gamma/\alpha = \delta/\beta$ not equal\footnote{This will be automatically satisfied if $\Pi$ is Klingen-ordinary at $p$, since this ratio then has $p$-adic valuation $r_2 + 2 > 0$.} to $-1$, so that the vector $w^{\Kl\prime}_{\alpha\beta} \in \cW(\pi)^{\Kl(p)}$ of \eqref{eq:wklprime} is defined. As Proposition 5.6 of \cite{LPSZ1} we can extend this to a compatible collection of vectors $w^{\Kl\prime}_n\in \cW(\pi)^{\Kl(p^n)}$ for $n \ge 1$, all of them eigenvectors for the operators $U_{2, \Kl}'$ with eigenvalue $\alpha\beta/p^{r_2 + 1}$, and satisfying
   \[ w^{\Kl,\prime}_n = \frac{1}{p^3} \sum_{g \in \Kl(p^n) / \Kl(p^{n+1})} g \cdot w^{\Kl\prime}_{n+1}\quad \forall\, n \ge 1. \]

   We define the following Euler factor:
    \begin{align*}
     \cE(\pi, m) &\coloneqq
     \left(1 - \tfrac{p^m}{\alpha}\right)\left(1 - \tfrac{p^m}{\beta}\right)
     \left(1 - \tfrac{\gamma}{p^{m+1}}\right)\left(1 - \tfrac{\delta}{p^{m+1}}\right).
    \end{align*}

    We write $w^{\sph}_\tau$ for the spherical Whittaker function of the $\GL_2$ representation $\tau$, normalised so that $w^{\sph}_\tau(1) = 1$; the values of this function along the maximal torus are given by
    \[
     w^{\sph}_\tau\left( \stbt{p^n}{0}{0}{1} \right)= p^{-n(r_1 + r_2 + 4)/2} \left(\alpha^n + \alpha^{n-1}\beta+ \dots + \beta^n\right).
    \]

    \begin{proposition}
     \label{prop:whittakerformula}
     Let $\Phi_1, \Phi_2$ be Schwartz functions on $\Qp^2$, with $\Phi_1$ satisfying $\Phi_1'(0,0) = 0$, where $\Phi_1'$ is the partial Fourier transform in the second variable only. Then for all $m \gg 0$  we have
     \begin{multline*}
      \wZ(u_{\Kl} \cdot w^{\Kl,\prime}_m, \Phi_1 \times \Phi_2) =
     \\  \frac{\cE(\pi,  q) \cE(\pi \times \chi_2^{-1}, r_2 + 1 + r)}{(1+p^{-1})^2(1 - p^{-1})}\int_{\Qp^\times} w^{\sph}_\tau(\stbt {x}{}{}1) W^{\Phi_1}(\stbt {x}{}{}1, \tfrac{-(r_1 - q - r)}{2}) W^{\Phi_2}(\stbt {x}{}{}1, \tfrac{-(r_2 - q + r)}{2})\tfrac{\theta(x)}{|x|} \, \mathrm{d}^\times x,
     \end{multline*}
     where as above $\theta$ is the unramified character mapping $p$ to $\gamma/\alpha = \beta/\delta$.
    \end{proposition}

    \begin{remark}
     Here we are assuming that the integrand has no pole at the relevant value of $(s_1, s_2)$, which can only happen if $r = (r_1 - r_2 + 1)/2$ and  $\cE(\pi \times \chi_2^{-1}, r_2 + 1 + r)$ vanishes. If the integrand does have a pole then the formula remains valid if interpreted via analytic continuation (either in $r$, or in the Hecke parameters).
    \end{remark}

    \begin{proof}
     This is a special case of Proposition 8.14 of \cite{LPSZ1}. Since we are assuming $\pi$ and the $\chi_i$ to be unramified, the epsilon-factor term in \emph{op.cit.} is 1; and the ratio of $L$-factors gives the two $\cE(\pi, -)$ terms. Moreover, as shown above, our renormalisation of the Klingen test vectors is precisely the one which scales the (non-normalised) $\GL_2$ Whittaker function $W_{\xi}$ of \emph{op.cit.} to its normalised equivalent $w_\tau^{\sph}$.
    \end{proof}

  \subsubsection{Particular cases}
   \label{sect:whittakerintegral}
  We define the following Schwartz functions on $\Qp^2$:
  \begin{itemize}
   \item $\Phi'_{\dep} = \ch( \Zp^\times \times \Zp^\times )$,
   \item $\Phi'_{\crit} = \ch( \Zp \times \Zp^\times)$,
  \end{itemize}
  These will correspond to holomorphic Eisenstein series that are respectively $p$-depleted, or critical-slope (hence the notation). We let $\Phi_{?}$ denote the preimage of $\Phi'_{?}$ under the inverse Fourier transform (in the second variable only); these are a little messy to write down explicitly.

  Then we have the following formulae, assuming $n \ge 0$ and $\chi$ unramified:
  \begin{itemize}
   \item $W^{\Phi_{\dep}}\left( \stbt{p^n}{}{}{1}, s \right) = $ 1 if $n = 0$, and zero otherwise.
   \item $W^{\Phi_{\crit}}\left( \stbt{p^n}{}{}{1}, s \right) = p^{-ns}$,
  \end{itemize}

  Accordingly, for these data the integral appearing in \cref{prop:whittakerformula} is given by
  \begin{multline*}
   \int_{\Qp^\times} w^{\sph}_\tau(\stbt {x}{}{}1) W^{\Phi_1}(\stbt {x}{}{}1, \tfrac{-(r_1 - q - r)}{2}) W^{\Phi_2}(\stbt {x}{}{}1, \tfrac{-(r_2 - q + r)}{2})\tfrac{\theta(x)}{|x|} \, \mathrm{d}^\times x \\= 
   \begin{cases}
    1, \quad\text{if $(\Phi_1, \Phi_2) = (\Phi_{\dep}, \Phi_{\crit})$, $(\Phi_{\crit}, \Phi_{\dep})$ or $(\Phi_{\dep}, \Phi_{\dep})$;} \\[2mm]
    \left[\left(1 - \frac{\gamma}{p^{1 + q}}\right)\left(1 - \frac{\delta}{p^{1 + q}}\right)\right]^{-1},
    \quad \text{if $(\Phi_1, \Phi_2) = (\Phi_{\crit} , \Phi_{\crit}).$}
   \end{cases}
  \end{multline*}

   \begin{remark}
    The case of $(\Phi_{\dep}, \Phi_{\dep})$ appears already in the computation of the interpolating property of the $p$-adic $L$-function in \cite{LPSZ1}. The case of $(\Phi_{\dep}, \Phi_{\crit})$ will appear in our formula for the syntomic regulator; and the case of $(\Phi_{\crit} , \Phi_{\crit})$ plays a somewhat different role -- it is used in \cref{ssec:testdataatp} to compare the \'etale class for $\Pi$ at prime-to-$p$ level to an auxiliary \'etale class at $\Kl(p)$ level which is easier to study.
   \end{remark}

  \subsection{Twisted zeta integrals}
   \label{sect:twistedzeta}

   To deal with certain ``junk'' terms that arise in the evaluation of the $p$-adic regulator,  we will need to consider a twisted form of the above integrals. We let $\chi_1, \chi_2, \rho$ be smooth characters of $\Qp^\times$ with $\chi_1 \chi_2 = \rho^2 \chi_{\pi}$, and consider the slightly more general integral
   \[ Z(w, \uPhi; s_1, s_2, \uchi, \rho)\coloneqq \int_{(Z_G N_H\backslash H)(\Qp)} w(h) f^{\Phi_1}(h_1; \chi_1, s_1) f^{\Phi_2}(h_2; \chi_2, s_2) \rho(\det h)\ \mathrm{d}h.\]
   This is, of course, an instance of the $\GSp_4$ zeta-integral of \cite[Definition 8.3]{LPSZ1} with $\pi$ replaced by $\pi \times \rho$; but we want to focus on the case where $\pi$ is unramified (as above) but $\rho$ and the $\chi_i$ are not, so it is helpful to also consider it as an instance of the $\GSp_4 \times \GL_2$ zeta-integral of \emph{op.cit.} for $\pi \times \sigma$, where $\sigma$ is taken to be the representation
   \[ \sigma \coloneqq I(|\cdot|^{s_2 - 1/2} \rho, |\cdot|^{1/2 - s_2}\rho \chi_2^{-1})\]
   (subject to an appropriate definition of the Whittaker model of $\sigma$ in the reducible case, as in the footnote to Proposition 8.14 of \emph{op.cit.}). The first interpretation shows that the fractional ideal generated by the values of the renormalised zeta-integral
   \[ \wZ(\dots) = \frac{Z(\dots)}{L(\pi \times \rho, s_1 + s_2 - \tfrac{1}{2}) L(\pi \times \rho \chi_2^{-1}, s_1 - s_2 + \tfrac{1}{2})}\]
   is the unit ideal of $\CC[\ell^{\pm s_1}, \ell^{\pm s_2}]$. The second interpretation shows that if $\Phi_1'(0,0) = 0$, then for $m \gg 0$ we have the special-value formula
   \begin{multline*}
   \wZ(u_{\Kl} \cdot w^{\Kl,\prime}_m, \Phi_1 \times \Phi_2, s_1, s_2; \uchi, \rho) =
   \frac{1}{L(\tau \times \sigma \times \theta, s_1) L(\tau^\vee \times \sigma^\vee, 1-s_1) \epsilon(\tau \times \sigma, s_1)} \\
   \times
   \tfrac{p^3}{(p+1)^2(p-1)}\int_{\Qp^\times} w^{\sph}_\tau(\stbt {x}{}{}1) W^{\Phi_1}(\stbt {x}{}{}1; \chi_1, s_1) W^{\Phi_2}(\stbt {x}{}{}1; \chi_2, s_2)\tfrac{\theta(x)\rho(x)}{|x|} \, \mathrm{d}^\times x.
   \end{multline*}
   In this more general setting, the test functions we shall use are of the form
   \begin{align*}
    \Phi'_{\dep, \mu, \nu}(x, y) &= \ch( \Zp^\times \times \Zp^\times ) \cdot \mu(x) \nu(y),
    & \Phi'_{\crit, \nu}(x,y) &= \ch( \Zp \times \Zp^\times)\cdot \nu(y),\\
    \langle p\rangle^{-1}\varphi \cdot \Phi'_{\crit, \nu} &= \ch(p\Zp \times \Zp^\times) \cdot \nu(y).
   \end{align*}
   for finite-order characters $\mu, \nu$, with $\chi|_{\Zp^\times} = \mu^{-1}\nu$ (taking $\mu$ to be trivial in the case of $\Phi'_{\crit, \nu}$, so this condition becomes simply $\chi|_{\Zp^\times} =\nu$). Note that $\Phi'_{\dep, \mu, \nu}(x, y)$ is the same function considered in \cite[Definition 7.5]{LPSZ1}.

   We have
   \[ W^{\Phi_{\dep, \mu, \nu}}(\stbt{x}{}{}{1}; \chi, s)
    =
   \begin{cases}
   \mu(-x)\nu(-1) & \text{if $x \in \Zp^\times$}\\
    0 & \text{otherwise}
   \end{cases}
   \]
   and
   \[ W^{\Phi_{\crit, \nu}}(\stbt{x}{}{}{1}; \chi, s)
   =
   \begin{cases}
   |x|^s \nu(-1) & \text{if $x \in \Zp$}\\
   0 & \text{otherwise.}
   \end{cases}
   \]
   Thus, for test data $\Phi_{\dep, \mu_1, \nu_1} \times \Phi_{\dep, \mu_2, \nu_2}$, assuming $\mu_1 \nu_1 \mu_2 \nu_2 = 1$ and $\rho = \nu_1 \nu_2$, the torus Whittaker integral is simply the integral of the constant function 1 over $\Zp^\times$, so it is 1. Similarly, for test data of the form $\Phi_{\dep, \mu_1, \nu_1} \times \Phi_{\crit, \nu_2}$ we again obtain that the integral is 1. In particular, the values for ``$\dep \times \dep$'' and ``$\dep \times \crit$'' test data are equal.

   On the other hand, the Whittaker function of $\langle p\rangle^{-1}\varphi\Phi_{\crit, \nu}$ is zero at $\stbt x 0 0 1$ unless $x\in p\Zp$, so if we consider test data of the form $\langle p\rangle^{-1}\varphi \cdot \Phi_{\crit, \nu_1} \times \Phi_{\dep, \mu_2, \nu_2}$, then we are integrating the product of a function supported on $p\Zp$ and another supported on $\Zp^\times$. Hence the zeta integral is 0.

\section{Formulating a reciprocity law}

 We now formulate a precise relation between the canonical class $z^{[\Pi, q, r]}_{\can}(\uchi)$ and values of $p$-adic $L$-functions. We continue with the same notation and assumptions on $\Pi$ as in the previous section.

 \subsection{Hecke parameters at $p$}
  \label{sect:heckeparams}

  Assume $\Pi$ and the characters $\chi_i$ are unramified at $p$, and write $w \coloneqq r_1 + r_2 + 3$.

  \begin{definition} \
   \begin{itemize}
    \item We define the \emph{Hecke polynomial} of $\Pi'$ at $p$ to be the degree 4 polynomial $P_p(X)$ such that
    \[ L(\Pi_p', s - \tfrac{3}{2}) = L(\Pi_p, s - \tfrac{r_1 + r_2 + 3}{2}) = P_p(p^{-s})^{-1}. \]
    \item The \emph{Hecke parameters} of $\Pi'$ at $p$ are the complex numbers $\alpha, \beta, \gamma, \delta$ such that
    \[
     P_p(X) = (1 - \alpha X) (1 - \beta X) (1 - \gamma X) (1 - \delta X),\qquad
     \alpha \delta = \beta \gamma = p^{(r_1 + r_2 + 3)} \chi_{\Pi}(p).
    \]
   \end{itemize}
  \end{definition}

  If $E$ is any number field over which $\Pif$ is definable, then the coefficients of $P_p(X)$ lie in $\cO_E$; the Hecke parameters are algebraic integers in $\bar{E}$, and are well-defined up to the action of the Weyl group. Extending $E$ if necessary, we may assume that they lie in $\cO_E$ itself. All of the Hecke parameters have complex absolute value $p^{(r_1 + r_2 + 3)/2}$ (see \cite[Theorem 1]{weissauer05}).

  \begin{note}
   Our notations here for Hecke polynomials and Hecke parameters are consistent with the notations of \cite{LSZ17} (see Theorem 10.1.3 of \emph{op.cit.}~in particular). It is also consistent with \S 10 of \cite{LPSZ1}, where the main theorems of that paper are given. Note, however, that the Hecke parameters here are \emph{not} the same as the $(\alpha, \beta, \gamma, \delta)$ in \cite{LPSZ1} Proposition 3.2, which are the Hecke parameters of a different twist of $\Pi_p$. We apologise to readers of \cite{LPSZ1} for shifting normalisations in the middle of the paper.
  \end{note}

  We shall fix an embedding $E \into L \subset \QQbar_p$, where $L$ is a finite extension of $\Qp$, and let $v_p$ be the valuation on $L$ such that $v_p(p) = 1$. If we order $(\alpha, \beta, \gamma, \delta)$ in such a way that $v_p(\alpha) \le \dots \le v_p(\delta)$ (which is always possible using the action of the Weyl group), then we have the valuation estimates
  \begin{equation}
   \label{eq:valuations}
   v_p(\alpha) \ge 0,\qquad v_p(\alpha \beta) \ge r_2 + 1.
  \end{equation}

  \begin{remark}\label{rem:Dcrisevals}
   These inequalities correspond to the fact that the Newton polygon of the $p$-adic Galois representation associated to $\Pi$ lies on or above the Hodge polygon; see \cref{prop:Dcrisevals} below.
  \end{remark}

  \begin{definition}
   We say $\Pi$ is \emph{Siegel ordinary} at $p$ if $v_p(\alpha) = 0$; \emph{Klingen ordinary} at $p$ if $v_p(\alpha\beta) = r_2 + 1$; and \emph{Borel ordinary} if it is both Siegel and Klingen ordinary.
  \end{definition}

 \subsection{Exponential maps and regulators}

  \begin{convention}
  	The representation $\Qp(1)$ of $\Gal(\QQbar_p / \Qp)$ has Hodge--Tate weight $1$, and crystalline Frobenius $\varphi$ acts on $\Dcris(\Qp(1))$ as multiplication by $1/p$.
  \end{convention}

  Recall the following result (which follows from Theorem 1 of \cite{urban05}, together with the multiplicity-one results of Arthur and Gee--Ta\"ibi cited above):

  \begin{proposition}
   \label{prop:Dcrisevals}
   The representation $V_{\Pi} |_{G_{\Qp}}$ is crystalline. The eigenvalues of $\varphi$ on $\Dcris(V_{\Pi})$ are the Hecke parameters $\{ \alpha, \beta, \gamma, \delta\}$ of \cref{sect:heckeparams}, and its Hodge--Tate weights are $\{0, -r_2-1, -r_1-2, -r_1-r_2-3\}$.\qed
  \end{proposition}

  \begin{lemma}\label{lemma:h1g}
   For all $0 \le q \le r_2$, we have the following:
   \begin{enumerate}
    \item[(a)] The operators $1 - \varphi$ and $1 - p\varphi$ are bijective on $\Dcris(V_{\Pi}^*(-q))$.
    \item[(b)] The Bloch--Kato $H^1_{\mathrm{e}}$, $H^1_{\mathrm{f}}$ and $H^1_{\mathrm{g}}$ subspaces of $H^1(\Qp, V_{\Pi}^*(-q))$ coincide.
    \item[(c)] The Bloch--Kato exponential map
    \[   \exp: \frac{\DdR(V_{\Pi}^*)}{\Fil^{-q} \DdR(V_{\Pi}^*)} \to H^1_{\mathrm{e}}(\Qp, V_{\Pi}^*(-q))\]
    is an isomorphism.
   \end{enumerate}
  \end{lemma}

  \begin{proof}
   It is well known (see e.g. \cite[\S 3]{blochkato90}) that (a) implies (b) and (c). The assertion (a) amounts to claiming that
   \[ \{ \alpha, \beta, \gamma, \delta\} \cap \{ 1, p, \dots, p^{r_2 + 1}\} = \varnothing.\]
   However, all elements in the first set have Archimedean absolute value $p^{(r_1 + r_2 + 3)/2}$, and since $r_1 \ge r_2$, we have $(r_1 + r_2 + 3)/2 > r_2 + 1$.
  \end{proof}

  Since the localisation at $p$ of the class $ z^{[\Pi, q, r]}_{\can}$ is in $H^1_{\mathrm{g}}$ (by \cite[Theorem B]{nekovarniziol16}), it is also in $H^1_{\mathrm{e}}$. Letting $\log$ denote the inverse of the Bloch--Kato exponential, we may define
  \[ \log\left(  z^{[\Pi, q, r]}_{\can} \right) \in \frac{\DdR(V_{\Pi}^*)}{\Fil^{-q} \DdR(V_{\Pi}^*)} = \left(\Fil^1 \DdR(V_{\Pi}) \right)^*.\]
  Note that the target of this map is 3-dimensional (and independent of $q$ in this range).

  \begin{assumption}
   We assume henceforth that $\Pi$ is Klingen-ordinary at $p$.
  \end{assumption}

  \begin{lemma}
   \label{lem:trivzero}
   If $\Pi$ is Klingen-ordinary at $p$, then none of $(\alpha, \beta, \gamma, \delta)$ has the form $p^n \zeta$ with $n \in \ZZ$ and $\zeta$ a root of unity. (In other words, Assumption 11.1.1 of \cite{LSZ17} is satisfied.)
  \end{lemma}

  \begin{proof}
   Since all of the Hecke parameters are Weil numbers of weight $w = r_1 + r_2 + 3$, it follows that if one of the parameters has this form, then $w$ must be even and $n = w/2$. In particular, this parameter has $p$-adic valuation $w/2$. However, if $\Pi$ is Klingen-ordinary then $\alpha, \beta$ have valuations at most  $r_2 + 1 \le \tfrac{w-1}{2}$, and $\gamma, \delta$ have valuations at least $r_1 + 2 \ge \tfrac{w+1}{2}$, so none can have valuation $w/2$.
  \end{proof}

  In this Klingen-ordinary setting there is a distinguished pair $(\alpha, \beta)$ of Hecke parameters of minimal valuation, and hence a distinguished 2-dimensional subspace
  \begin{equation} \label{eq:defcurlyP}
   \Dcris(V_{\Pi})^{\cQ(\varphi) = 0}, \qquad \cQ(t) = (1 - \tfrac{t}{\alpha})(1 - \tfrac{t}{\beta}).
  \end{equation}

  \begin{note}\label{note:1diml}
   From weak admissibility, we see that $\Dcris(V_{\Pi})^{\cQ(\varphi) = 0} \cap \Fil^1$ must have dimension exactly 1, and that it surjects onto the 1-dimensional graded piece $\Fil^1 / \Fil^{r_2 + 2}$.
  \end{note}

  \begin{definition}
   \label{def:nudR}
   Let $\nu$ be a basis of the 1-dimensional $L$-vector space $\Gr^{r_2 + 1} \DdR(V_{\Pi})$, and let $\nu_{\dR}$ denote its unique lifting to $\Dcris(V_{\Pi})^{\cQ(\varphi) = 0} \cap \Fil^{r_2 + 1}$.
  \end{definition}

  We can now formulate the key problem treated in this paper and its sequel:
  \bigskip
  \begin{mdframed}
   \textbf{Problem}: Compute the quantity
   \begin{equation}
    \label{eq:goal0}
    \operatorname{Reg}_{\nu, \can}^{[\Pi, q, r]}(\uchi) \coloneqq \left\langle \nu_{\dR}, \log\left( z^{[\Pi, q, r]}_{\can}(\uchi)\right)\right\rangle_{\DdR(V_{\Pi})} \in L.
   \end{equation}
  \end{mdframed}
  \bigskip

  By definition, for any $(w, \uPhi)$ with $\uPhi$ in the $\uchi^{-1}$-eigenspace, we have
  \[ \left\langle \nu_{\dR}, \log\left( z^{[\Pi, q, r]}(w, \uPhi) \right)\right\rangle_{\DdR(V_{\Pi})} = \wZ(w, \uPhi) \cdot \operatorname{Reg}_{\nu, \can}^{[\Pi, q, r]}(\uchi). \]
  So it suffices to evaluate the quantity
  \[ \operatorname{Reg}_{\nu}^{[\Pi, q, r]}(w, \uPhi) \coloneqq \left\langle \nu_{\dR}, \log\left( z^{[\Pi, q, r]}(w, \uPhi) \right)\right\rangle_{\DdR(V_{\Pi})}\]
  for a single $(w, \uPhi)$ with $\wZ(w, \uPhi) \ne 0$.
%


 \subsection{Periods and $p$-adic $L$-functions}
 \label{sect:periods}

  We shall relate the regulator $\operatorname{Reg}_{\nu}^{[\Pi, q, r]}(\uchi)$ to the $p$-adic $L$-functions of \cite{LPSZ1}; so let us briefly recall the construction of \emph{op.cit.} (and slightly refine it by paying closer attention to the periods involved).

  \subsubsection*{P-adic periods} Recall that we have chosen a basis vector $\nu$ of the 1-dimensional $L$-vector space $\frac{\Fil^1 \DdR(V_{\Pi})}{\Fil^{r_2 + 2} \DdR(V_{\Pi})}$. This space is canonically the base-extension to $L$ of an $E$-vector space, namely
  \[ \Hom_{E[G(\Af)]}\Big( \cW(\Pif'), H^2(\Pif)_E\Big),\]
  where $H^2(\Pif)_E$ denotes the unique copy of $\Pif'$ inside a coherent $H^2$ of a toroidal compactification of $Y_{G, E}$, as in \cite[\S 5.2]{LPSZ1}.

  \begin{definition}
   Let $\nu^{\alg}$ be an $E$-basis of $\Hom_{E[G(\Af)]}\Big( \cW(\Pif'), H^2(\Pif)_E\Big)$, and let $\Omega_p(\Pi, \nu, \nu^{\alg}) \in L^\times$ be the scalar such that we have
   \[ \nu = \Omega_p(\Pi, \nu, \nu^{\alg}) \cdot \nu^{\alg}.\]
  \end{definition}

  \begin{remark}
   We could, of course, suppose $\nu$ to be $E$-rational, so that we could take $\nu^{\alg} = \nu$ and $\Omega_p(\Pi, \nu, \nu^{\alg}) = 1$; but it is convenient to allow more general $\nu$ in order to allow variation in $p$-adic families.
  \end{remark}

  \subsubsection*{Archimedean periods} Let $K_\infty \subseteq G(\mathbf{R})$ denote the standard maximal connected compact-mod-centre subgroup $\mathbf{R}^\times U_2(\mathbf{R})$. As in \cite[\S 2.0]{harris90b}, the standard Shimura datum for $\GSp_4$ determines a parabolic subalgebra $\mathfrak{p} \subseteq \operatorname{Lie} G$ with Levi subalgebra $\operatorname{Lie}K_\infty$. 

  Theorem 5.2.3 of \cite{harris90b} gives an isomorphism between the $(\mathfrak{p}, K_\infty)$-cohomology of the space of cusp forms on $G(\QQ) \backslash G(\AA)$ and a certain space of of harmonic differential forms, which injects into the interior part of coherent cohomology of the toroidal compactification (Proposition 3.6 of \emph{op.cit.}). In our situation this gives a $G(\Af)$-equivariant injection
  \[ H^2(\mathfrak{p}, K_\infty; \Pi' \otimes V_\sigma) \into H^2(\Pif)_{\CC} \]
  for a suitable algebraic representation $V_\sigma$ depending on $(r_1, r_2)$. As in Proposition 4.5 of op.cit., the $(\mathfrak{p}, K_\infty)$-cohomology can be computed rather more simply as
  \[ H^2(\mathfrak{p}, K_\infty; \Pi' \otimes V_\sigma) \cong \Hom_{K_\infty}(\tau, \Pi') \]
  where $\tau$ is the minimal $K_\infty$-type\footnote{This notion is most familiar for holomorphic discrete series, but can be extended to the non-holomorphic discrete series representations considered here; see \cite{moriyama04} for further details.} of $\Pi_\infty'$ (whose highest weight is the same as that of $V_\sigma^\vee \otimes \bigwedge^2(\mathfrak{p}^-)$). See \S 2.2 of \cite{harriskudla92} for explicit formulae in the $\GSp_4$ case. Since $\tau$ appears in $\Pi'_\infty$ with multiplicity one, both $\Hom_{K_\infty}(\tau, \Pi')$ and $H^2(\Pif)_{\CC}$ are isomorphic to $\Pif'$, which is irreducible; so the above injection must be a bijection, giving a canonical isomorphism
  \[ H^2(\Pif)_{\CC} \cong \Hom_{K_\infty}(\tau, \Pi').\]

  \begin{remark}
   See Theorem 1.2 of \cite{harriskudla92} for a somewhat stronger statement under an additional regularity hypothesis on the weight, showing that the subspaces $H^2(\Pif)$, as $\Pi$ varies over cuspidal automorphic representations with a given discrete-series Archimedean component, span the whole of the interior cohomology of the toroidal compactification. Without this additional regularity assumption, there might be some non-tempered automorphic representations which also contribute to the coherent cohomology; but (from Arthur's classification) these cannot contribute to the $\Pif'$-eigenspace.
  \end{remark}

  We also have a canonical isomorphism $\Pi' \cong \cW(\Pi')$, the global Whittaker transform, mapping $\phi \in \Pi$ to the Whittaker function given by
  \[ g \mapsto \int_{N(\QQ) \backslash N(\AA)} \phi(ng) \psi_N(n)^{-1} \mathrm{d}n \]
  where $N$ is the unipotent radical of $B$. Composing this with the comparison isomorphism above, we obtain an isomorphism
  \[  H^2(\Pif)_\CC \xrightarrow{\ \cong\ } \Hom_{K_\infty}(\tau, \cW(\Pi')).\]

  In \cite[\S 10.2]{LPSZ1}, we define an element $w_\infty \in \Hom_{K_\infty}(\tau, \cW(\Pi_{\infty}'))$, mapping the standard basis vectors of $\tau$ to certain explicit Whittaker functions in $\cW(\Pi_\infty')$ constructed by Moriyama in \cite{moriyama04} (normalised such that the associated zeta-integral is exactly equal to the Archimedean $L$-factor).

  \begin{proposition}
   There is a constant $\Omega_\infty(\Pi, \nu^{\alg}) \in \CC^\times$ such that the composite map
   \[ \cW(\Pif') \xrightarrow{\ \nu^{\alg}\ } H^2(\Pif)_{\CC} \longrightarrow \Hom_{K_\infty}(\tau, \Pi') \longrightarrow \Hom_{K_\infty}(\tau, \cW(\Pi')) \]
   maps $w_\mathrm{f} \in \cW(\Pif')$ to $\Omega_\infty(\Pi, \nu^{\alg})^{-1} \cdot w_{\mathrm{f}} \otimes w_\infty$.
  \end{proposition}

  \begin{proof} Since all of these maps are $G(\Af)$-equivariant bijections between representations isomorphic to $\Pif'$, and $\Pif'$ is irreducible, it is clear that the composite is multiplication by a nonzero scalar.\end{proof}

  \begin{remark}
   The quantities $\Omega_p(\Pi, \nu, \nu^{\alg})$ and $\Omega_\infty(\Pi, \nu^{\alg})$ each depend on the choice of $\nu^{\mathrm{alg}}$, but the ratio
   \[ \Omega_p(\Pi, \nu, \nu^{\alg})^{-1} \otimes \Omega_\infty(\Pi, \nu^{\alg}) \in L \otimes_{E} \CC\] depends only on $\nu$; the dependency on $\nu^{\alg}$ cancels out.
  \end{remark}

  \subsubsection*{The $p$-adic $L$-function}

  \begin{theorem}
   \label{thm:padicLfcn}
   There exists a $p$-adic measure $\cL_{p, \nu}(\Pi, \uchi) \in \Lambda_L(\Zp^\times \times \Zp^\times)$ whose evaluation at $(a_1 + \rho_1, a_2 + \rho_2)$, for $a_i$ integers with $0 \le a_1, a_2 \le r_1 - r_2$ and $\rho_i$ finite-order characters such that we have $ (-1)^{a_1 + a_2} \rho_1(-1) \rho_2(-1) = -\chi_2(-1)$, satisfies
   \begin{multline*}
    \frac{\cL_{p, \nu}(\Pi, \uchi; a_1 + \rho_1, a_2 + \rho_2)}{\Omega_p(\Pi, \nu, \nu^{\alg})}
    =\\ R_p(\Pi, \rho_1, a_1) R_p(\Pi \times \chi_2^{-1}, \rho_2, a_2) \cdot
    \frac{\Lambda(\Pi \times \rho_1^{-1}, \tfrac{1-r_1+r_2}{2} + a_1)\Lambda(\Pi \times \chi_2^{-1}\rho_2^{-1}, \tfrac{1-r_1+r_2}{2} + a_2)}{\Omega_{\infty}(\Pi, \nu^{\alg})},
   \end{multline*}
   for some (and hence every) choice of $\nu^{\alg}$ as above.
  \end{theorem}

  Here $\Lambda(\Pi \times \rho, s)$ denotes the completed $L$-function (including its Archimedean $\Gamma$-factors). This was proved in \cite{LPSZ1} assuming $r_2 \ge 1$, and taking $\uchi = (\chi, \id)$. See Proposition 10.3 of \emph{op.cit.} for further details, including the definition of the factors $R_p(\Pi, \rho, a)$. The generalisation to arbitrary $\uchi$ is immediate; to extend the results to $r_2 = 0$, we note that the only place that the assumption $r_2 \ge 1$ is needed in \cite{LPSZ1} is in order to show that the Klingen-ordinary cohomology is concentrated in a single degree. If $r_2 = 0$ then this fails, but one still obtains a complex concentrated in degrees $\{0, 1\}$, and since we only care about the highest-degree cohomology this does not affect the proof of our main theorem.

  (Note also that the interpolating property was only proved in \emph{op.cit.}~under the assumption that $a_1 \ge a_2$, but by comparing interpolating properties at points with $a_1 = a_2$, one sees easily that both sides are unchanged if we swap the roles of $(a_1, \rho_1, \chi_1)$ and $(a_2, \rho_2, \chi_2)$, so this condition can be removed.)

 \subsection{Statement of the regulator formula}

  We can now give a precise statement of the theorem we shall prove:
  \bigskip

  \begin{mdframed}
   \begin{theorem}[to be proved]
    \label{thm:mainthm}
    For any $q, r$ with $0 \le q \le r_2$, $0 \le r \le r_1 - r_2$, and $(-1)^{r_2 - q + r} = \chi_2(-1)$, the regulator of \cref{eq:goal0} is given by
    \[
     \operatorname{Reg}^{[\Pi, q, r]}_{\nu, \can}(\uchi) = \frac{(-2)^q (-1)^{r_2 - q+1} (r_2 - q)!}{\cE_p(\Pi, q) \cE_p(\Pi \times \chi_2^{-1}, r_2 + 1 + r)} \cdot \cL_{p, \nu}(\Pi, \uchi; -1-r_2 + q, r),
    \]
    where $\cE_p(\Pi, n) \coloneqq \left(1- \tfrac{p^n}{\alpha}\right) \left(1- \tfrac{p^n}{\beta}\right) \left(1- \tfrac{\gamma}{p^{n+1}}\right) \left(1- \tfrac{\delta}{p^{n+1}}\right)$ (which is nonzero for all $n \in \ZZ$ by Lemma \ref{lem:trivzero}).

    Equivalently, for all test data $(w, \uPhi)$ with $\uPhi$ in the $\uchi^{-1}$-eigenspace, we have
    \begin{equation}
     \label{eq:goal}
     \operatorname{Reg}^{[\Pi, q, r]}_{\nu}(w, \uPhi) = \frac{(-2)^q (-1)^{r_2 - q+1} (r_2 - q)!}{\cE_p(\Pi, q) \cE_p(\Pi, r_2 + 1 + r)} \cdot \cL_{p, \nu}(\Pi, \uchi; -1-r_2 + q, r) \cdot \wZ(w, \uPhi).
    \end{equation}
   \end{theorem}
  \end{mdframed}

  \begin{note} \
   \begin{enumerate}[(a)]
    \item The factor $\cE_p(\Pi, n)$ agrees (up to a sign) with $R_p(\Pi, \id, -1-r_2 + n)$; that is, the Euler factors relating $\cL_{p, \nu}$ to the regulator $\mathrm{Reg}_{\nu, \can}^{[\Pi, q, r]}$ in the ``geometric'' range are formally the same as those relating it to complex $L$-values in the ``critical'' range.

    \item If $r_1 - r_2 > 0$, or if Hypothesis 10.5 of \cite{LPSZ1} holds, then $\cL_{p, \nu}(\Pi, \bfj_1, \bfj_2)$ factors as a product of a function of $j_1$ and a function of $j_2$. However, our proof of the theorem will not directly ``see'' this finer decomposition.\qedhere
   \end{enumerate}
  \end{note}

 \subsection{Test data at $p$}\label{ssec:testdataatp} We now choose specific test data at the prime $p$ which will enable us to evaluate both sides of \eqref{eq:goal}.

  \begin{definition} \
   \begin{itemize}
    \item As above, let $u_{\Kl}\in G(\Zp)$ be any matrix whose first column is $(1,1,0,0)^T$.
    \item Let $\Kl(p)$ denote the Klingen parahoric subgroup (as in ``Conventions'' above).
    \item Let $U_{2, \Kl}'$ denote the operator $p^{-r_2} \left[ \Kl(p) \diag(1, p, p, p^2) \Kl(p) \right] \in \QQ[ G(\Qp)/\!/ \Kl(p)]$.
    \item Let $w_{p, \Kl} \in \cW(\Pi'_p)^{\Kl(p)}$ denote the normalised $U_{2, \Kl}'$-eigenvector of eigenvalue $\tfrac{\alpha\beta}{p^{r_2 + 1}}$ defined in \eqref{eq:wklprime} above.
    \item Let $\Phi_{\crit}$ be the Schwartz function on $\Qp^2$ defined in \cref{sect:whittakerintegral}, and $\uPhi_{p, \Kl} = \Phi_{\crit} \times \Phi_{\crit}$.
   \end{itemize}
   We refer to the pair $(u_{\Kl} \cdot w_{p, \Kl}, \uPhi_{p, \Kl})$ as \emph{Klingen test data}.
  \end{definition}

  We already evaluated $\wZ_p(u_{\Kl} \cdot w_{p, \Kl}, \uPhi_{p, \Kl})$ in \cref{sect:whittakerintegral} above; the result is
  \[
   \wZ_p(u_{\Kl} \cdot w_{p, \Kl}, \uPhi_{p, \Kl}) = \frac{p^3}{(p+1)^2(p-1)} \cdot \frac{\cE(\Pi, q) \cE(\Pi, r_2 + 1 + r)}{\left(1 - \frac{\gamma}{p^{1+q}}\right)\left(1 - \frac{\delta}{p^{1+q}}\right)}.
  \]
  In particular, it is nonzero; so it suffices to prove that \eqref{eq:goal} holds for test data $(w, \uPhi)$ given by the product of these Klingen test data at $p$, and some arbitrary test data $(w^p, \uPhi^p)$ away from $p$. We shall now make explicit the two sides of  \eqref{eq:goal} in this setting.

  \subsubsection{Etale side}

  The left-hand side of \cref{eq:goal} for data of this form can be written explicitly using \cref{prop:explicitZ}. Let us choose an open compact $U^p$ such that $U^p$ fixes $w^p$ and $U^p \cap H(\Af^p) = V^p$ fixes $\uPhi^p$.

  \begin{notation}
    Write $Y_{G, \Kl,\QQ}$ for the $G$-Shimura variety of level $U^p \Kl(p)$, and $Y_{H, \Delta,\QQ}$ for the $H$-Shimura variety of level $V^p K_{p, \Delta}$, where
  \[
   K_{p,\Delta}=\left\{ h\in H(\Zp):\, h=\left(\begin{pmatrix} x & \star\\  & \star\end{pmatrix},\begin{pmatrix} x & \star\\  & \star\end{pmatrix}\right)\pmod{p}\quad\text{for some $x$}\right\}.
  \]
  Write $X_{G,\Kl,\QQ}$ and $X_{H,\Delta,\QQ}$ for toroidal compactifications, where the rational polyhedral cone decompositions are chosen as in \cite[\S 2.2.4]{LPSZ1}.
  \end{notation}

  We have $u_{\Kl}^{-1} K_{p, \Delta} u_{\Kl} \subset \Kl(p)$, so as in \cite[\S 4.1]{LPSZ1}, $u_{\Kl}$ gives a finite morphism of Shimura varieties
  \begin{equation}
   \label{eq:iotaDelta}
   \iota_{\Delta}: Y_{H, \Delta,\QQ} \to Y_{G, \Kl,\QQ}.
  \end{equation}
  Hence there is a pushforward map $\iota_{\Delta, \star}^{[t_1, t_2]}$ on \'etale cohomology, and we obtain a class
  \[
   \left(\log \circ \AJ^{[\Pi, q]} \circ \iota^{[t_1,t_2]}_{\Delta, \star}\right)(\Eis^{[t_1, t_2]}_{\et, \Phi}) \in H^3_{\dR}( Y_{G, \Kl, \Qp}, \cV^\vee) \otimes L / \Fil^{-q} .
  \]
  On the other hand, the element $\nu$ of \cref{def:nudR} restricts to a homomorphism
  \[ \cW(\Pif')^{\Kl(p)}_L \to \Gr^{r_2 + 1} H^3_{\dR, c}( Y_{G, \Kl, L}, \cV), \]
  so we have a class $\eta = \nu(w)$ in the target group, and a canonical lifting $\eta_{\dR} = \nu_{\dR}(w)$ of $\eta$ to $\Fil^{1+r_2} H^3_{\dR, c}( Y_{G, \Kl, L}, \cV) = \Fil^{1+q} H^3_{\dR, c}( Y_{G, \Kl, \Qp}, \cV) \otimes L$. From the defining property of the regulator map $\operatorname{Reg}_{\nu}^{[\Pi, q, r]}$, we have
  \[
   \operatorname{Reg}_{\nu}^{[\Pi, q, r]}(w, \uPhi) = \vol(V) \cdot \left\langle \left(\log \circ \AJ^{[\Pi, q]} \mathop{\circ} \iota^{[t_1,t_2]}_{\Delta, \star}\right)(\Eis^{[t_1, t_2]}_{\et, \Phi}), \eta_{\dR}\right\rangle_{\dR, Y_{G, \Kl, \Qp}}.
  \]

  \subsubsection{Coherent side} We now derive a corresponding formula for the right-hand side of \eqref{eq:goal}. The toroidal compactification of the Shimura variety $X_{G,\Kl,\QQ}$ (for a suitable choice of boundary data) has a canonical $\Zp$-model $X_{G,\Kl}$, and we let $\mathfrak{X}_{G,\Kl}$ denote its $p$-adic completion, as a formal scheme over $\Zp$.

  Given $\uPhi^p$, the construction of \cite{LPSZ1} \S 7.4 gives a 2-parameter $p$-adic family of Eisenstein series on $H$, which we denote simply by $\cE(\uPhi^p)$. Then the $p$-adic interpolation theory of \emph{op.cit.} allows us to make sense of $\iota_{\Delta, \star}\left(\cE(\uPhi^p)\right)$ as a class in $H^1$ of the multiplicative locus\footnote{This was denoted $X_{G, \Kl}^{\ge 1}$ in \emph{op.cit.}, but this notation is somewhat misleading since this space is only one component of the $p$-rank $\ge 1$ locus at Klingen level, so we shall use the above notations here (matching the notations used in \cite{LZ26}).} $\mathfrak{X}_{G, \Kl}^m \subset \mathfrak{X}_{G, \Kl}$. This class takes values in a sheaf of $\Lambda_L(\Zp^\times \times \Zp^\times)$-modules, and hence allows us to define a measure
  \[
   \left\langle \iota_{\Delta, \star}\left(\cE(\uPhi^p)\right), \eta \right\rangle_{\mathfrak{X}_{G, \Kl}^{m}} \in \Lambda_L(\Zp^\times \times \Zp^\times).
  \]
  (Here $\Lambda_L(-)$ denotes the Iwasawa algebra with $L$-coefficients.) This cup product depends on the choice of the prime-to-$p$ level group $U^p$, but this can be eliminated by renormalising by $\vol V$. Unwinding the definition of the $p$-adic $L$-function given in \cite{LPSZ1}, we reach the following explicit formula:

  \begin{proposition}
   \label{prop:klingendata}
   For $(q, r)$ as above, the value of the measure
   \[\vol(V)  \Big\langle \iota_{\Delta, \star}\left(\cE(\uPhi^p)\right),\eta \Big\rangle_{\mathfrak{X}_{G, \Kl}^{m}}\]
   at $(-1-r_2 + q, r)$ is $\tfrac{p^3}{(p+1)^2(p-1)} \wZ^p(w^p, \uPhi^p)\cL_{p, \nu}(\Pi, \uchi; -1-r_2 + q, r)$.\qed
  \end{proposition}

  Summarising the above discussion, we have the following:

  \begin{proposition}
   The formula of \eqref{eq:goal} is equivalent to the following assertion:

   For all prime-to-$p$ levels $U^p$, all $\uPhi^p$ stable under $U^p \cap H$, and all $\eta \in H^2(\Pif)^{U^p \Kl(p)}[U_{2, \Kl}' = \frac{\alpha\beta}{p^{r_2 + 1}}]$, we have
   \begin{multline}
    \label{eq:goal2}
    \left\langle \left(\log \circ \AJ^{[\Pi, q]} \mathop{\circ} \iota^{[t_1,t_2]}_{\Delta, \star}\right)(\Eis^{[t_1, t_2]}_{\et, \uPhi^p \uPhi_{\Kl}}),\eta_{\dR}\right\rangle_{\dR, Y_{G, \Kl, \Qp}}
    \\ = \frac{(-2)^q (-1)^{r_2 - q+1}(r_2 - q)!}
    {\left(1 - \frac{\gamma}{p^{1+q}}\right)\left(1 - \frac{\delta}{p^{1+q}}\right)}
    \cdot \Big\langle \iota_{\Delta, \star}\left(\cE(\uPhi^p)|_{(-1-r_2 + q, r)}\right), \eta \Big\rangle_{\mathfrak{X}_{G, \Kl}^{m}}.
   \end{multline}
  \end{proposition}

  This is the formula that is actually proved in \cite{LZ26}.


\begin{thebibliography}{LPSZ21}

\bibitem[Art04]{arthur04}
\textsc{J.~Arthur}, \emph{Automorphic representations of
  {$\operatorname{GSp}(4)$}}, Contributions to automorphic forms, geometry, and
  number theory, Johns Hopkins Univ. Press, 2004, pp.~65--81. \MR{2058604}

\bibitem[BK90]{blochkato90}
\textsc{S.~Bloch} and \textsc{K.~Kato},
  \articlehref{http://doi.org/10.1007/978-0-8176-4574-8}{\emph{{$L$}-functions
  and {T}amagawa numbers of motives}}, The {G}rothendieck {F}estschrift,
  {V}ol.\ {I} (\textsc{P.~Cartier} et~al.{}, eds.), Progr. Math., vol.~86,
  Birkh{\"a}user, Boston, MA, 1990, pp.~333--400. \MR{1086888}

\bibitem[Bum97]{bump97}
\textsc{D.~Bump},
  \articlehref{http://doi.org/10.1017/CBO9780511609572}{\emph{Automorphic forms
  and representations}}, Cambridge Studies in Advanced Mathematics, vol.~55,
  Cambridge Univ. Press, 1997. \MR{1431508}

\bibitem[BG14]{buzzardgee14}
\textsc{K.~Buzzard} and \textsc{T.~Gee},
  \articlehref{http://doi.org/10.1017/CBO9781107446335.006}{\emph{The
  conjectural connections between automorphic representations and {G}alois
  representations}}, Automorphic forms and {G}alois representations. {V}ol. 1,
  London Math. Soc. Lecture Notes, vol. 414, Cambridge Univ. Press, 2014,
  pp.~135--187. \MR{3444225}

\bibitem[CS80]{CS80}
\textsc{W.~Casselman} and \textsc{J.~Shalika},
  \articlehref{https://www.numdam.org/item/CM_1980__41_2_207_0/}{\emph{The
  unramified principal series of {$p$}-adic groups. {II.} {The} {Whittaker}
  function}}, Compos. Math. \textbf{41} (1980), no.~2, 207--231. \MR{581582}

\bibitem[ET23]{emorytakeda21}
\textsc{M.~Emory} and \textsc{S.~Takeda},
  \articlehref{http://doi.org/10.1007/s00209-023-03228-3}{\emph{Contragredients
  and a multiplicity one theorem for general spin groups}}, Math. Z.
  \textbf{303} (2023), no.~3, Paper No. 70, 54. \MR{4550488}

\bibitem[GT11]{gantakeda11}
\textsc{W.~T. Gan} and \textsc{S.~Takeda},
  \articlehref{http://doi.org/10.4007/annals.2011.173.3.12}{\emph{The local
  {L}anglands conjecture for {$\operatorname{GSp}(4)$}}}, Ann. of Math. (2)
  \textbf{173} (2011), no.~3, 1841--1882. \MR{2800725}

\bibitem[GT19]{geetaibi18}
\textsc{T.~Gee} and \textsc{O.~Ta\"{\i}bi},
  \articlehref{http://doi.org/10.5802/jep.99}{\emph{Arthur's multiplicity
  formula for {$\operatorname{GSp}_4$} and restriction to
  {$\operatorname{Sp}_4$}}}, J. \'Ecole Poly. Math. \textbf{6} (2019),
  469--535. \MR{3991897}

\bibitem[GT05]{genestiertilouine05}
\textsc{A.~Genestier} and \textsc{J.~Tilouine}, \emph{Syst\`emes de
  {T}aylor--{W}iles pour {$\operatorname{GSp}_4$}}, Formes automorphes. II. Le
  cas du groupe {$\operatorname{GSp}(4)$}, Ast{\'e}risque, vol. 302, Soc. Math.
  France, 2005, pp.~177--290. \MR{2234862}

\bibitem[Har90]{harris90b}
\textsc{M.~Harris},
  \articlehref{http://doi.org/10.4310/jdg/1214445036}{\emph{Automorphic forms
  of {$\overline\partial$}-cohomology type as coherent cohomology classes}}, J.
  Differential Geom. \textbf{32} (1990), no.~1, 1--63. \MR{1064864}

\bibitem[HK92]{harriskudla92}
\textsc{M.~Harris} and \textsc{S.~Kudla},
  \articlehref{http://doi.org/10.1215/S0012-7094-92-06603-8}{\emph{Arithmetic
  automorphic forms for the nonholomorphic discrete series of
  {$\operatorname{GSp}(2)$}}}, Duke Math. J. \textbf{66} (1992), no.~1,
  59--121. \MR{1159432}

\bibitem[HS01]{harrisscholl01}
\textsc{M.~Harris} and \textsc{A.~J. Scholl},
  \articlehref{http://doi.org/10.1007/s100970000026}{\emph{A note on trilinear
  forms for reducible representations and {B}eilinson's conjectures}}, J. Eur.
  Math. Soc. \textbf{3} (2001), no.~1, 93--104. \MR{1812125}

\bibitem[KLZ20]{KLZ20}
\textsc{G.~Kings, D.~Loeffler,} and \textsc{S.~L. Zerbes},
  \articlehref{http://doi.org/10.1353/ajm.2020.0002}{\emph{{R}ankin--{E}isenstein
  classes for modular forms}}, Amer. J. Math. \textbf{142} (2020), no.~1,
  79--138.

\bibitem[Loe21]{loeffler-ggp}
\textsc{D.~Loeffler},
  \articlehref{http://doi.org/10.1515/forum-2021-0089}{\emph{Gross--{P}rasad
  periods for reducible representations}}, Forum Math. \textbf{33} (2021),
  no.~5, 1169--1177. \MR{4308634}

\bibitem[Loe24]{loeffler-zeta1}
\bysame, \articlehref{https://nyjm.albany.edu/j/2024/30-1.html}{\emph{On local
  zeta-integrals for {$\GSp(4)$} and {$\GSp(4) \times \GL(2)$}}}, New York J.
  Math. \textbf{30} (2024), 1--23.

\bibitem[LPSZ21]{LPSZ1}
\textsc{D.~Loeffler, V.~Pilloni, C.~Skinner,} and \textsc{S.~L. Zerbes},
  \articlehref{http://doi.org/10.1215/00127094-2021-0049}{\emph{Higher {H}ida
  theory and {$p$}-adic {$L$}-functions for {$\operatorname{GSp}(4)$}}}, Duke
  Math. J. \textbf{170} (2021), no.~18, 4033--4121. \MR{4348233}

\bibitem[LSZ22]{LSZ17}
\textsc{D.~Loeffler, C.~Skinner,} and \textsc{S.~L. Zerbes},
  \articlehref{http://doi.org/10.4171/JEMS/1124}{\emph{Euler systems for
  {$\operatorname{GSp}(4)$}}}, J. Eur. Math. Soc. \textbf{24} (2022), no.~2,
  669--733. \MR{4382481}

\bibitem[LZ26]{LZ26}
\textsc{D.~Loeffler} and \textsc{S.~L. Zerbes},
  \articlehref{http://arxiv.org/abs/2003.05960}{\emph{On the {B}loch--{K}ato
  conjecture for {$\operatorname{GSp}(4)$}}}, Cambridge J. Math. (2026), to
  appear, \path{arXiv:2003.05960}.

\bibitem[MW12]{moeglinwaldspurger12}
\textsc{C.~M{\oe}glin} and \textsc{J.-L. Waldspurger}, \emph{La conjecture
  locale de {G}ross--{P}rasad pour les groupes sp\'{e}ciaux orthogonaux: le cas
  g\'{e}n\'{e}ral}, Ast{\'e}risque \textbf{347} (2012), 167--216, Sur les
  conjectures de Gross et Prasad. II. \MR{3155346}

\bibitem[Mor04]{moriyama04}
\textsc{T.~Moriyama},
  \articlehref{http://doi.org/10.1353/ajm.2004.0032}{\emph{Entireness of the
  spinor {$L$}-functions for certain generic cusp forms on
  {$\operatorname{GSp}(2)$}}}, Amer. J. Math. \textbf{126} (2004), no.~4,
  899--920. \MR{2075487}

\bibitem[NN16]{nekovarniziol16}
\textsc{J.~Nekov{\'a}{\v{r}}} and \textsc{W.~Nizio{\l}},
  \articlehref{http://doi.org/10.2140/ant.2016.10.1695}{\emph{Syntomic
  cohomology and {$p$}-adic regulators for varieties over {$p$}-adic fields}},
  Algebra \& Number Theory \textbf{10} (2016), no.~8, 1695--1790, with
  appendices by Laurent Berger and Fr\'ed\'eric D\'eglise. \MR{3556797}

\bibitem[Oka19]{okazaki}
\textsc{T.~Okazaki}, \articlehref{http://arxiv.org/abs/1902.07801}{\emph{Local
  {W}hittaker-newforms for {$\operatorname{GSp}(4)$} matching to {L}anglands
  parameters}}, preprint, 2019, \path{arXiv:1902.07801}.

\bibitem[Pra96]{prasad96}
\textsc{D.~Prasad}, \articlehref{http://doi.org/10.1007/BF01446282}{\emph{Some
  applications of seesaw duality to branching laws}}, Math. Ann. \textbf{304}
  (1996), no.~1, 1--20. \MR{1367880}

\bibitem[RS07]{robertsschmidt07}
\textsc{B.~Roberts} and \textsc{R.~Schmidt},
  \articlehref{http://doi.org/10.1007/978-3-540-73324-9}{\emph{Local newforms
  for {$\operatorname{GSp}(4)$}}}, Lecture Notes in Math., vol. 1918, Springer,
  Berlin, 2007. \MR{2344630}

\bibitem[RW18]{roesnerweissauer18}
\textsc{M.~R\"osner} and \textsc{R.~Weissauer},
  \articlehref{http://arxiv.org/abs/1810.09419}{\emph{Spinor {E}uler factors
  for {$\operatorname{GSp}(4)$} in the subregular case}}, preprint, 2018,
  \path{arXiv:1810.09419}.

\bibitem[RW25]{roesnerweissauer17}
\bysame, \articlehref{http://doi.org/10.2140/pjm.2025.337.137}{\emph{Regular
  poles for spinor {$L$}-series attached to split {B}essel models of
  {$\operatorname{GSp}(4)$}}}, Pacific J. Math. \textbf{337} (2025), no.~1,
  137--199.

\bibitem[ST93]{sallytadic93}
\textsc{P.~J. Sally, Jr.} and \textsc{M.~Tadi\'{c}}, \emph{Induced
  representations and classifications for {$\operatorname{GSp}(2,F)$} and
  {$\operatorname{Sp}(2,F)$}}, M{\'e}m. Soc. Math. France (N.S.) \textbf{52}
  (1993), 75--133. \MR{1212952}

\bibitem[Urb05]{urban05}
\textsc{E.~Urban}, \emph{Sur les repr\'esentations {$p$}-adiques associ\'ees
  aux repr\'esentations cuspidales de {$\operatorname{GSp}_4 / \QQ$}}, Formes
  automorphes. II. Le cas du groupe {$\operatorname{GSp}(4)$}, Ast{\'e}risque,
  vol. 302, Soc. Math. France, 2005, pp.~151--176. \MR{2234861}

\bibitem[Wei05]{weissauer05}
\textsc{R.~Weissauer}, \emph{Four dimensional {G}alois representations}, Formes
  automorphes. II. Le cas du groupe {$\operatorname{GSp}(4)$}, Ast{\'e}risque,
  vol. 302, Soc. Math. France, 2005, pp.~67--150. \MR{2234860}

\end{thebibliography}

\providecommand{\bysame}{\leavevmode\hbox to3em{\hrulefill}\thinspace}
\providecommand{\MR}[1]{%
 MR \href{http://www.ams.org/mathscinet-getitem?mr=#1}{#1}.
}
\providecommand{\href}[2]{#2}
\newcommand{\articlehref}[2]{\href{#1}{#2}}

\end{document}